\documentclass[12pt]{article}
\usepackage{pdfsync}
\usepackage{amsmath,amssymb}
\usepackage{bm}

\usepackage{amssymb}
\usepackage{amsmath}
\usepackage{amsbsy}
\usepackage{multibib}
\usepackage{latexsym}
\usepackage{cancel}
\usepackage{amsthm}
\usepackage{color,graphicx}
\usepackage{subfigure}
\usepackage{placeins}
\newtheorem{theorem}{Theorem}[section]

\newtheorem{lemma}[theorem]{Lemma}
\newtheorem{proposition}[theorem]{Proposition}
\newtheorem{corollary}[theorem]{Corollary}

\newcommand{\nref}[1]{(\ref{#1})}

\newcommand{\RR}{{\mathbb R}}

\newcommand{\bbeta}{\bm {\eta}}

\newcommand{\norm}[1]{\|#1\|}
\newcommand{\snorm}[1]{| #1 |}

\newcommand{\normAh}[1]{\norm{#1}_{\mathcal{A}}}

\newcommand{\Shnorm}[1]{\lVert#1\rVert_{\traceh}}
\newcommand{\ShnormDG}[1]{\lVert#1\rVert_{\traceh,*}}

\newcommand{\trace}{\Phi} %{T}
\newcommand{\traceh}{\trace_h}
\newcommand{\tracehk}{\trace_h^{\ell}}

\newcommand{\Tk}{\Phi_h^\ell}

\newcommand{\bx}{{\mathbf x}}
\newcommand{\bu}{{\mathbf u}}
\newcommand{\bg}{{\mathbf g}}

\newcommand{\bA}{{\mathbf A}}

\newcommand{\bF}{{\mathbf F}}
\newcommand{\bI}{{\mathbf I}}
\newcommand{\bK}{{\mathbf K}}
\newcommand{\bJ}{{\mathbf J}}
\newcommand{\bM}{{\mathbf M}}
\newcommand{\bP}{{\mathbf P}}
\newcommand{\bR}{{\mathbf R}}
\newcommand{\bS}{{\mathbf S}}

\newcommand{\bQ}{{\mathbf Q}}

\newcommand{\bXi}{\mathbf{\Xi}}

\newcommand{\ii} {\mathtt{i}}
\newcommand{\ie} {\mathtt{e}}
\newcommand{\iv} {\mathtt{v}}
\newcommand{\iin} {\mathtt{n}}

\newcommand{\ush}{u^*_h}

\newcommand{\skel}{{\Sigma}}

\renewcommand{\lor }{\longrightarrow}

\newcommand{\Rh}{R_h}
\newcommand{\shat}{{\hat s}}

\newcommand{\sE}{\shat^E}
\newcommand{\sV}{\shat^V}
\newcommand{\etao}{\eta_{o}}
\newcommand{\ai}{A_{\ell}}
\newcommand{\bi}{B_{\ell}}
\newcommand{\Ndom}{N}
\newcommand{\NE}{N^\ell_E}

\newcommand{\omk}{{\Omega_{\ell}}}

\newcommand{\Hk}{{H_{\ell}}}

\newcommand{\Gk}{{\Gamma_{\ell}}}
\newcommand{\Gkb}{{\Gamma^{\partial}_{\ell}}}

\newcommand{\Lineari}{{\mathfrak L}_{H}}

\newcommand{\h} {\mathtt{h}}
\newcommand{\p} {\mathtt{p}}

\newcommand{\ho} {h_{\ell}} % {\mathtt{h}_{\ell}} %{\mathtt{h}_{\ast}}
\newcommand{\po} {p_{\ell}}  %{p_{\ast}}

\newcommand{\Ho}{H_{\ell}}  %{H_{\ast}}

\newcommand{\faclog}{\left( 1 + \log\left(H\,p^2/ h\right) \right)}
\newcommand{\faclogi}{\left( 1 + \log\left(\frac{\Ho\,\po^2}{\ho} \right)
  \right)}

\newcommand{\gki}{{E_{\ell}^i}}

\newcommand{\Vk}{{X_h^{\ell}}}
\newcommand{\Xh}{{X_h}}
\newcommand{\Xho}{{X^{0}_h}}

\newcommand{\jump}[1]{\lbrack\!\lbrack\,#1\,\rbrack\!\rbrack} %usata
\newcommand{\Th}{{\mathcal T}_{h}} %usata
\newcommand{\Thk}{\mathcal{T}^\ell_h}
\newcommand{\Ts}{{\mathcal T}_{H}} %{\mathcal{T}_S} %usata
\newcommand{\re}{\mathbb{R}}
\newcommand{\K}{K} %{T} %usata
\newcommand{\Eh}{{{\mathcal E}_h}} %usata
\newcommand{\calA}{{\mathcal A}} %usata
\newcommand{\Eho}{{{\mathcal E}^{o}_h}} %usata

\newcommand{\Ehb}{{{\mathcal E}^{\partial}_h}} %usata
\newcommand{\Ehi}{{{\mathcal E}^{\ell}_h}} %usata
\newcommand{\Ehbi}{{{\mathcal E}^{\partial,\ell}_h}} %usata
\newcommand{\Ehio}{{{\mathcal E}^{o,\ell}_h}} %usata
\newcommand{\vect}[1]{\bf #1}
\newcommand{\n}{{\vect{n}}} %usata
\newcommand{\taub}{{\vect{\tau}}} %usata}

\newcommand{\ziz}{\zeta_0}
\newcommand{\lin}{\zeta_L}
\newcommand{\Iuno}{I_1}
\newcommand{\Idue}{I_2}
\newcommand{\aai}{{a_i}}
\newcommand{\bbi}{{b_i}}

\newcommand{\dd}[1]{\, \textrm{d#1}}

\newcommand{\av}[1]{\{#1\}} %usata
\renewcommand{\O}{\Omega} %usata
\newtheorem{remark}{Remark}[section]

\title{Substructuring Preconditioners for \\
an $h$-$p$ Nitsche-type method}
\author{P.F. Antonietti\thanks{
MOX, Dipartimento di Matematica,
Politecnico di Milano, Piazza Leonardo da Vinci 32, I-20133 Milano, Italy.
{\tt paola.antonietti@polimi.it}} ,
B. Ayuso de Dios\thanks{
Centre de Recerca Matem\'atica
Campus de Bellaterra, Edifici C
08193 Bellaterra (Barcelona), Spain. {\tt bayuso@crm.cat}} ,
S. Bertoluzza\thanks{
Istituto di Matematica Applicata e Tecnologie Informatiche del CNR, Via Ferrata, 1, 27100 Pavia, Italy. {\tt Silvia.Bertoluzza@imati.cnr.it }} ,
M. Pennacchio\thanks{
Istituto di Matematica Applicata e Tecnologie Informatiche del CNR, Via Ferrata, 1, 27100 Pavia, Italy. {\tt Micol.Pennacchio@imati.cnr.it}}
}
\date{ } %\today}

\begin{document}
\maketitle

\begin{abstract}
We propose and study an iterative substructuring method for an $h$-$p$ Nitsche-type discretization, following the original  approach introduced in \cite{BPS} for conforming methods. We prove quasi-optimality with respect to the mesh size and the polynomial degree for the proposed preconditioner. Numerical experiments asses the performance of the preconditioner and verify the theory.
\end{abstract}

\section{Introduction}
Discontinuous Galerkin (DG) Interior Penalty (IP) methods were introduced in the late 70's for approximating elliptic problems. They were arising as a natural {\it evolution} or extension of Nitsche's method \cite{nitsche0},  and were based on the observation that inter-element continuity could be attained by penalization; in the same spirit  Dirichlet boundary conditions are weakly imposed for Nitsche's method \cite{nitsche0}.
The use, study and application of DG IP methods was abandoned for a while, probably due to the fact that they were never proven to be more advantageous or efficient than their conforming relatives. The lack of optimal and efficient solvers for the resulting linear systems, at that time, surely was also contributing to that situation. 

However, over the last 10-15 years, there has been a considerable interest in the development and understanding of DG methods for elliptic problems (see, for instance, \cite{abcm} and the references therein), partly due to the simplicity with which the DG methods handle non-matching grids and allow for the  design of hp-refinement strategies. The IP and Nitsche approaches have also found some new applications; in the design of new conforming and non-conforming methods \cite{abm0,abfm0, mika0, dominik0,guido-riviere, hansbo-elasticity} and as a way to deal with non-matching grids for domain decomposition \cite{Stenberg.2003,duarte1}.

This has also motivated the interest in developing efficient solvers for DG methods.  In particular, additive Schwarz methods are considered and analyzed  in \cite{FengKarakashian01,BrennerWang05,AntoniettiAyuso2007,AntoniettiAyuso2008,AntoniettiAyuso2009,AntoniettiHouston2011,BarkerBrennerParkSung2011,BrennerWang05}. Multigrid methods are studied in \cite{GopalakrishnanKanschat2003,BrennerZhao_2005,Kanschat2003}. 
Two-level methods and multi-level methods are presented
in~\cite{Dobrev_et_al_2006,dahmen1} and other subspace correction methods are considered in \cite{AyusoZikatanov2009,jump,AyusoB_GeorgievI_KrausJ_ZikatanovL-2009aa}. 

Still the development of preconditioners for DG methods based on Domain Decomposition (DD)  has been mostly limited to classical Schwarz methods. Research towards more sophisticated non-overlapping DD preconditioners, such as the BPS ({\it Bramble Pasciak Schatz}), Neuman-Neuman, BDDC, FETI or FETI-DP is now at its inception.  Non-overlapping DD methods typically refer to methods defined on a decomposition of a domain made up of a collection of mutually disjoint subdomains, generally called {\it substructures}. These family of methods  are obviously well suited for parallel computations and furthermore, for several problems (like problems with jump coefficients) they offer some advantages over their relative overlapping methods, and have already proved their usefulness. Roughly speaking, these methods are  algorithms  for preconditioning the Schur complement with respect to the unknowns on the skeleton of the subdomain partition.  They are generally referred {\it substructuring preconditioners}.

While the theory for the  confoming case is now well established and
understood for many problems \cite{Toselli:2004:DDM}, the
discontinuous nature of the finite element spaces at the interface of
the substructures (in the case of Nitsche-type methods) or even within
the skeleton of the domain partition, poses extra difficulties in the
analysis which preclude from having a straight extension of such
theory. Mainly, unlike in the conforming case, the coupling of the
unknowns along the interface does not allow for splitting the global
bilinear form as a sum of local bilinear forms associated to the
substructures (see for instance \cite{FengKarakashian01} and
\cite[Proposition 3.2]{AntoniettiAyuso2007}). 
Moreover the discontinuity of the finite element space makes the use
of standard $H^{1/2}$-norms in the analysis of the discrete harmonic
functions difficult.
% On the other hand, again the discontinuity of the finite element space, difficults the classical use of $H^{1/2}$-norms in the analysis of the discrete harmonic functions. 

For Nitsche-type methods, a new definition of discrete harmonic function has been introduced in \cite{Dryja.Galvis.Sarkis.2007} together with some tools (similar to those used in the analysis of mortar preconditioners) that allow them to adapt and extend the  general theory \cite{Toselli:2004:DDM} for  substructuring preconditioners in two dimensions. More precisely,  in  \cite{Dryja.Galvis.Sarkis.2007, sarkis1, sarkis2} the authors introduced and analyzed  {\it Balancing Domain with Constrains}  BDDC, Neuman-Neuman and FETI-DP domain decomposition preconditioners for a first order Nitsche type discretization of an elliptic problem with jumping coefficients. For the discretization, a symmetric IP DG scheme is used (only) on the skeleton of the subdomain partition, while piecewise linear conforming approximation is used in the interior of the subdomains. 
In these works, the authors prove quasi-optimality  with respect to the mesh-size and optimality with respect to the jump in the coefficient. They also address  the case of non-conforming meshes. 

More recently, several BDDC preconditioners have been introduced and analyzed for some full DG discretizations \cite{schoberl, luca, eun00}, following a different path. 
In  \cite{schoberl} the authors consider the $p$-version of the
preconditioner for an Hibridized IP DG method \cite{hybrid0,hybrid2},
for which the unknown is defined directly on the skeleton of the
partition. They prove cubic logarithmic growth on the polynomial
degree but also show numerically that the results are not sharp. The
IP DG and the  IP-spectral DG methods for an elliptic problem with
jumping coefficient are considered in \cite{eun00} and \cite{luca},
respectively. In both works the approach for the analysis differs
considerably from the one taken in \cite{Dryja.Galvis.Sarkis.2007, sarkis1,sarkis2} and relies on suitable space decomposition of the global DG space; using either nonconforming or conforming subspaces. This allow the authors to adapt  the classical theories  for analyzing the resulting BDDC preconditioners.

In this work, we focus on the original substructuring approach introduced in \cite{BPS} for conforming discretization of two dimensional problems and in \cite{BPSIV,Dryja.Smith.Widlund94} for three dimensions (see also  \cite{Xu.Zou,Toselli:2004:DDM} for a detailed description).  In the framework of non conforming domain decomposition methods,
this kind of preconditioner has been applied to the mortar method \cite{AMW,BertPenn,Pennacchio.08,Pennacchio.Simoncini.08} and to the three fields domain
decomposition method \cite{Bsubstr}, always considering the
$h$-version of the methods.  
For spectral discretizations and the $p$ version of conforming approximations the preconditioner has been studied in \cite{pavarino1,mandel2}. For $h$-$p$ conforming discretizations of two dimensional problems the BPS preconditioner is studied in \cite{ainsworthBPS}. To the best of our knowledge, this preconditioner has not been considered for Nitsche or DG methods before. 

Here, we propose a BPS (Bramble-Pasciak-Schatz) preconditioner for an $h$-$p$ Nitsche type discretization of elliptic problems. In our analysis, we use some of the tools introduced in \cite{Dryja.Galvis.Sarkis.2007,sarkis1}, such as their definition of the discrete harmonic lifting that allows for defining the discrete Steklov-Poincar\'e operator associated to the Nitsche-type method. However, our construction of the preconditioners is guided  by  the definition of a  suitable norm on the skeleton of the subdomain partition, that scales like an $H^{1/2}$-norm and captures the energy of the DG functions on the skeleton. This allow us  to provide a much simpler analysis,  proving quasi-optimality with respect to the mesh size and the polynomial degree for the proposed preconditioners. 
Furthermore, we demonstrate that unlike what happens in the conforming
case, to ensure quasi-optimality of the preconditioners a block
diagonal structure that de-couples completely the edge and vertex
degrees of freedom on the skeleton is not possible; 
this is due to the presence of the penalty term which is needed to deal
with the discontinuity.
%the discontinuous nature on the skeleton of the considered finite element space.
We show however that the implementation of the preconditioner can be done efficiently  and that it  performs in agreement with the theory.

The rest of the paper is organized as follows. The basic notation, functional setting and the  description of the Nitsche-type method are given in next section; Section~\ref{problem}. Some technical tools required in the construction and analysis of the proposed preconditioners are revised in Section~\ref{sec:technical_tool}.
The substructuring preconditioner is introduced and analyzed in Section~\ref{sec:substr}. Its  practical implementation together with some variants of the preconditioner are discussed in Section~\ref{sec:variants}. The theory is verified through several numerical experiments presented in Section~\ref{sec:expes}. The proofs of some technical lemmas used in our analysis are reported in the Appendix \ref{app0}.

\section{Nitsche methods and Basic Notation}\label{problem}

In this section, we introduce the basic notation, the functional setting and the Nitsche discretization.

To ease the presentation we restrict ourselves to the following model problem.
Let $\Omega\subset \mathbb{R}^{2}$ a bounded polygonal domain, let $f\in L^2(\Omega)$ and let
\begin{equation*}
\left\{\begin{aligned}
-\Delta u^{\ast} &= f \qquad&& \mbox {in   } \O,\\
u^{\ast}&=0\qquad&& \mbox {on   } \partial\O.
\end{aligned}\right.
\end{equation*}
The above problem admits the following weak formulation:
\smallskip
{\it find} $u^*\in H^1_0(\Omega)$ {\it such that}:
\begin{equation}\label{prob}
  a(u^*,v)=f(v)\qquad \text{ for all } v\in H^1_0(\Omega),
\end{equation}
where
\begin{eqnarray*}
a(u,v) =  \int_\Omega \nabla u \cdot \nabla v \dd{x} \qquad  f(v) =
\int_\Omega f v \dd{x}\;, \quad \forall\, u, v\in H^1_0(\Omega)
\end{eqnarray*}
\subsection{Partitions}
We now introduce the different partitions needed in our work.
We denote by $\Ts$ a subdomain partition of $\O$  into $\Ndom$ non-overlapping shape-regular triangular or quadrilateral subdomains:
\begin{equation*}
\begin{aligned}
&\overline{\O}=\bigcup_{\ell=1}^{\Ndom}  \overline{\Omega}_\ell ,
&&\omk \cap  \O_j =\emptyset
&& \ell\ne j.
&& \end{aligned}
\end{equation*}
We set
\begin{equation}\label{def:hls}
\Hk=\min_{j\,:\,  \bar \O_\ell \cap  \bar \O_j \ne \emptyset } H_{\ell,j} \qquad \mbox{where   }\quad H_{\ell,j}=\left|\partial\omk \cap \partial \O_j\right|\;,
 \end{equation}
 and we also assume that $\Hk\simeq \mbox{diam}(\omk)$ for each $\ell=1,\ldots N$. We finally
 define the granularity of $\Ts$ by  $H=\min_{\ell} \Hk$.
We denote by $\Gamma$  and $\Gamma^\partial$ respectively the interior and the boundary portions of the skeleton of the subdomain partition $\Ts$:
\begin{equation*}
\begin{aligned}
&\Gamma=\bigcup_{\ell=1}^{\Ndom} \Gk,
&& \Gk=\partial\omk\smallsetminus \partial\O
&& &\forall\, \ell =1,\ldots, \Ndom.\\
&\Gamma^{\partial}=\bigcup_{\ell=1}^{\Ndom} \Gkb,
&&\Gkb=:\partial\omk\cap \partial\O
&& &\forall\, \ell=1\ldots, \Ndom.
\end{aligned}
\end{equation*}
We also define the complete skeleton as $\Sigma=\Gamma \cup
\Gamma^{\partial}$. The edges of the subdomain partition that form the
skeleton will be denoted by $E\subset \Gamma$ and
% {\rosso
 we will refer to them as {\it macro edges}, if they do not allude to
 a particular subdomain or {\it subdomain edges}, when they do refer
 to a particular subdomain.
%}

For each $\omk$, let $\left \{\Th^\ell \right \}$ be a family of   {\it fine}
partitions of $\O_\ell$ into elements (triangles or quadrilaterals) $\K$
with diameter  $h_{\K}$. All partitions $\Th^\ell$ are assumed to be shape-regular and we
define a global partition $\Th$  of $\O$ as
\begin{equation*}
\Th =\bigcup_{\ell=1}^{\Ndom} \Th^\ell.
%\sim \prod \Th^\ell .
\end{equation*}

%{\rosso
Observe that by construction $\Th$ is a fine partition of $\Omega$ which is compatible within each subdomain $\O_\ell$ but which may be non matching across the skeleton $\Gamma$. Throughout the paper, we always assume that the following \emph{bounded local variation} property holds: for any pair of neighboring  elements
$K^{+} \in \Th^{\ell^{+}}$ and $K^{-} \in \Th^{\ell^{-}}$, $\ell^{+} \neq \ell^{-}$, $h_{K^{+}}\simeq h_{K^{-}}$. % See, however, Remark \ref{rem:bv} at the end of next section.

%}

Note that the restriction of $\Th$ to the skeleton $\Gamma$ induces a partition of each subdomain edge $E\subset \Gamma$.
We define the set of {\it element edges} on the skeleton
$\Gamma$  and on the boundary of $\O$ as follows:
\begin{align*}
\Eho:&=\{ e = \partial \K^{+} \cap \partial \K^{-} \cap \Gamma,  \,
\K^{+}\in \Th^{\ell^{+}}\, , \K^{-}\in \Th^{\ell^{-}}\, , \ell^{+} \neq \ell^{-}\} \,,&&   \\
\Ehb:&=\{ e= \partial \K \cap \partial \Omega,   \, \K \in \Th\} \,, &&
\end{align*}
and we set $\Eh=\Eho\cup \Ehb$.
When referring to a particular subdomain, say $\omk$ for some $\ell$, the set of element edges are denoted by
\begin{equation*}
\begin{aligned}
&\Ehio = \{ e \in \Eho:\ e \subset \partial\omk\}, \quad
&&\Ehbi=\{ e \in \Ehb:\ e \subset \partial \omk\}, \quad
&& \Ehi=\Ehio\cup \Ehbi\;.
\end{aligned}
\end{equation*}

\subsection{Basic Functional setting}

For $s\geq 1$, we define the broken Sobolev space

\begin{equation*}
H^{s}(\Ts)=\left\{ \phi \in
    L^{2}(\O) \,\,: \quad \phi\big|_{\omk} \in
    H^{s}(\omk) \quad \forall\,\, \omk \in \Ts \,\right\} \sim \prod_{\ell}\,  H^s(\omk)\;,
\end{equation*}
whereas the trace space associated to $H^{1}(\Ts)$ is defined by
\begin{eqnarray*}
\trace= \prod_{\ell}\, H^{1/2}(\partial \omk).
\end{eqnarray*}
For $u = (u^\ell)_{\ell=1}^\Ndom$ in $H^1(\Ts)$ we will denote by $u_{|_\Sigma}$
the unique element $\phi =(\phi^\ell)_{\ell=1}^\Ndom$ in $\Phi$ such that
$$ \phi^\ell =u^\ell_{|_{\partial \omk}}.$$

We now recall the definition of some trace operators following \cite{abcm}, and introduce the different discrete spaces that will be used in the paper.\\

Let $e\in \Eho$ be an edge on the interior skeleton shared by two elements $\K^{+}$ and $\K^{-}$ with outward unit normal vectors $\n^{+}$ and $\n^{-}$, respectively.
For scalar and vector-valued functions $\varphi \in
H^{1}(\Ts)$ and $\taub\in \left[{H}^{1}(\Ts)\right]^2$,
we define the \emph{average} and the \emph{jump} on $e \in \Eho$ as
\begin{align*}%\label{def-av}
&\av{\taub}=\frac {1} {2}(\taub^+ +\taub^-),
&&\jump{\varphi}=\varphi^+\n^++\varphi^-\n^-\;,
&&\mbox{on } e\in \Eho
\end{align*}
On a boundary element edge $e\in \Ehb$ we set
$\av{\taub}=\taub$ and $\jump{\varphi}=\varphi \n$, $\n$ denoting the outward unit normal vector to $\O$.\\

To each element $\K\in \Thk$, we associate a polynomial approximation order $p_{K} \geq 1$, and define the $hp$-finite element space of piecewise polynomials as
\begin{equation*}
\Vk= \{
v \in C^0(\omk)\text{~such~that~}v|_{\K} \in \mathbb{P}^{p_K}(\K), ~\K\in \Thk
\},
\end{equation*}
where $\mathbb{P}^{p_{K}}(\K)$ stands for the space of polynomials of
degree at most $p_{K}$ on $\K$.
%{\rosso
We also assume that the polynomial approximation order satisfies a
\emph{local bounded variation} property: for any pair of elements
$K^{+}$ and $K^{-}$ sharing an edge $e \in \Eho$, $p_{K^{+}} \simeq
p_{K^{-}}$. \\
%}

Our global approximation space $\Xh$ is  then defined as
\begin{equation*}
\Xh= \{ v\in L^{2}(\O) \,\, : \text{~such~that~}v_{|_{\omk}} \in \Vk \}\sim \prod_{\ell=1}^\Ndom \Vk\;,
.%  \ \subset H^{s}(\Ts).
\end{equation*}
We also define $\Xho \subset \Xh$ as the subspace of functions  of $\Xh$ vanishing on the skeleton $\Sigma$, i.e.,
\begin{equation*}
\Xho=\{ v\in \Xh \,\, : \text{~such~that~}v_{|_{\Sigma}} = 0 \} .
\end{equation*}
The trace spaces associated to $\Vk$ and $\Xh$  are defined as follows
\begin{align*}
\tracehk&=\{ \eta^{\ell} \in H^{1/2}(\partial\omk) \, :\,\,
\eta^{\ell} = w_{|_{ \partial\omk }} \text{ for some } w\in \Vk\} %{|_{ \partial\omk }} \}
%= \Vk_{|_{ \partial\omk }}\;
&&\forall \ell=1,\ldots ,\Ndom\;\\
 \traceh & = \prod_{\ell=1}^{\Ndom} \tracehk\;
  \  \subset \trace.
\end{align*}
Notice that the functions in the above finite element spaces are
conforming in the interior of each subdomain but are double-valued on $\Gamma$.
Moreover, any function $v\in \Xh$ can be represented as $v = ( v^{\ell})_{\ell=1}^\Ndom$
with $v^{\ell}\in \Vk$.\\

Next, for each subdomain $\omk\in \Ts$ and for each subdomain edge $E \subset \partial\omk$, we define the discrete trace spaces
 \begin{align*}
&\Phi_{\ell}(E) = {\tracehk}_{|_E}, %{\Phi_\ell}_{|_E},
&&\Phi^{o}_{\ell}(E)  = \{ \eta^{\ell}\in \Phi_{\ell}(E) \,\, :\; ~~ \eta^{\ell}=0  \;\mbox{on  }
 \partial E\, \}.
  \end{align*}
Note that, since we are in two dimensions, the boundary of a subdomain edge $E$ is the set of the two endpoints (or vertices) of $E$, that is if $E=(a,b)$ then $\partial E=\{a,b\}$.\\

Finally, we introduce a suitable coarse space $\Lineari \subset \trace$, that will be required  for the definition of the subtructuring
preconditioner:
\begin{equation}\label{eq:Lineari}
\Lineari = \{  \eta =(\eta^\ell) \in \trace \,\, : \,\,\,  \eta^\ell_{|_{E}} \in \mathbb{P}^{1}(E), \quad \forall\, E\subset \partial\omk\;, \,\,\, \forall \omk \in \Ts  \}\;.
\end{equation}

                                %

%%%%%%%%%%%%%%%%%%%%%%%%%%%%%%%%%%%%%%%%%%%%%%%%%%%%%%%%%%%%%%%%%%%%%%
\subsection{Nitsche-type methods}
In this section, we introduce the Nitsche-type method we consider for approximating the model problem \eqref{prob}. Here and in the following, to avoid the proliferation of constants, we will use the notation $x \lesssim y$ to represent the inequality $x \leq C y$, with $C > 0$ independent of the mesh size, of the polynomial approximation order, and of the size and number of subdomains. Writing $x \simeq y$ will signify that there exists a constant $C > 0$ such that $C^{-1}x \leq y \leq Cx$.\\

We introduce the local mesh size function $\h \in L^{\infty}(\Sigma)$ defined as
\begin{equation}\label{def:h}
\h(x)=\left\{
\begin{aligned}
&h_{K} \quad
&& \textrm{if $x\in\partial\K \cap \partial\Omega$},\\
&\min\{ h_{K^{+}},  h_{K^{-}}\}
&& \textrm{if $x\in\partial\K^{+} \cap \partial\K^{-}\cap\Gamma$,
$\K^{\pm}\in \Th^{\ell^{\pm}}\, , \ell^{+} \neq \ell^{-}$},
\end{aligned}\right.
\end{equation}
and the local  polynomial degree function $\p \in L^{\infty}(\Sigma)$:
 \begin{equation}\label{def:p2}
\p(x)=\left\{\begin{aligned}
&p_{K} \quad
&& \textrm{if $x\in\partial\K \cap \partial\Omega$},\\
&\max\{ p_{K^{+}},  p_{K^{-}}\}
&& \textrm{if $x\in\partial\K^{+} \cap \partial\K^{-}\cap \Gamma$,
$\K^{\pm}\in \Th^{\ell^{\pm}}$, $\ell^{+} \neq \ell^{-}$}.
\end{aligned}\right.
\end{equation}

\begin{remark}\label{las:h}
A different definition for the local mesh size function $\h$ and the local polynomial degree function $\p$  involving harmonic averages is sometimes used for the definition of Nitsche or DG methods \cite{Dryja.Galvis.Sarkis.2007}.
 We point out that such a definition yields to functions $\h$ and $\p$ which are of the same order as the ones given in \nref{def:h} and \nref{def:p2}, and therefore result in an equivalent method.
\end{remark}

We now define the following Nitsche-type discretization
\cite{Stenberg.1998,Stenberg.2003} to approximate problem \eqref{prob}:
find $\ush\in \Xh$ such that
\begin{equation}\label{discrete_problem}
  \calA_h(\ush,v_h)=f(v_h)\qquad \text{ for all } v_h\in \Xh\;,
\end{equation}
where, for all $u,v \in  \Xh$, $\calA_h(\cdot,\cdot)$ is defined as
\begin{equation}\label{def:nit}
\begin{aligned}
\calA_h(u,v)&=
\sum_{\ell=1}^{\Ndom}\int_{\omk} \nabla u \cdot \nabla v  \dd{x}
-\sum_{e \in \Eh} \int_e \av{\nabla u}\cdot \jump{v}\dd{s} &&\\
&\quad-\sum_{e \in \Eh} \int_e \jump{u}\cdot
\av{\nabla v}   \dd{s}
+\sum_{e \in \Eh} \alpha \,  \int_e \p^2  \, \h^{-1}\jump{u}\cdot\jump{v} \dd{s}. &&
\end{aligned}
\end{equation}

Here, $\alpha>0$ is the penalty parameter that needs to be
chosen $\alpha\geq \alpha_0$ for some $\alpha_0\gtrsim  1$ large
enough to ensure the coercivity of $\calA_h(\cdot,\cdot)$. \\

On $\Xh$, we introduce the following semi-norms:
\begin{equation}\label{eq:seminorms}
\begin{aligned}
|v|_{1,\Ts}^{2}=\sum_{\ell=1}^{\Ndom}\|\nabla v\|^{2}_{L^{2}(\omk)} ,
&&\qquad
 \snorm{v}_{\ast,\Eh}^{2} &=  \sum_{e\in \Eh} \|\p \, \h^{-1/2}\, \jump{v}\|_{L^{2}(e)}^{2} ,
\end{aligned}
\end{equation}
together with the natural induced norm by $\mathcal{A}_{h}(\cdot,\cdot)$:
\begin{equation}\label{def:normA}
\begin{aligned}
&\normAh{v}^2=|v|_{1,\Ts}^{2} + \alpha \snorm{v}_{\ast,\Eh}^{2}
&& \forall\, v\in \Xh\;.
\end{aligned}
\end{equation}
Following \cite{Stenberg.1998} (see also  \cite{abcm}) it is easy to
see the bilinear form $\calA_h(\cdot,\cdot)$ is continuous and
coercive (provided $\alpha\geq \alpha_0$ ) with respect the norm
\nref{def:normA}, i.e.,

\begin{equation*}
\begin{aligned}
&\textrm{Continuitiy}:
&&|\calA_h(u,v)| \lesssim \normAh{u}\normAh{v}
&&\forall\, u,v\in \Xh\\
&\textrm{Coercivity}:
&&\calA_h(v,v) \gtrsim \normAh{v}^{2}
&& \forall\, v\in \Xh.
\end{aligned}
\end{equation*}
From now on we will always assume that $\alpha \geq \alpha_0$.
Notice that the  continuity and coercivity constants depend only on
the shape regularity of $\Th$. 

%{\footnote{\blu Abbiamo tolto la vecchiaRemark 3.2. \`E commentata nel testo. Siamo convinte che anche senza l'ipotesi di bounded local variation il precondizionatore resti quasi ottimale, per cui scrivere che diventa suboptimal per noi non \`e corretto. A questo punto preferiamo non scrivere niente. Se il referee chieder\`a ci penseremo.} }

\section{Some technical tools}\label{sec:technical_tool}
We now revise some technical tools that will be required in the construction and analysis of the proposed preconditioners.\\

We recall the local {\em inverse inequalities} (cf. \cite{Schwab98}, for example):  for any $\eta \in \mathbb{P}^{p_{\K}}(K)$ it holds
\begin{equation*}
   | \eta |_{H^{r}(e)} \lesssim  p_{\K}^{2(r-s)}h_{\K}^{s - r} | \eta
 |_{H^{s}(e)}, \qquad e\subset \partial \K
 \end{equation*}
 for all $s,r$ with $0 \leq s < r \leq 1$.  Using the above inequality
 for $s = 0$  and space interpolation, it is easy to deduce that for a subdomain edge $E\subset \partial \Omega_\ell$ and for all $s,r$,  $0 \leq s < r < 1$, for all $\eta \in X_h^\ell|_E$ it holds that
 \begin{align}
    | \eta |_{H^{r}(E)}  &\lesssim
\max_{\substack{K \in \Th^\ell\\\partial K \cap E \neq \emptyset}}
(p_K^2 h_K^{-1})^{(r-s)}  | \eta |_{H^{s}(e)}
\lesssim \po^{2(s-r)} \ho^{s - r} | \eta|_{H^{s}(E)}, \label{inv:E} \\[4mm]
| \eta |_{H^{r}(\partial\omk)}
& \lesssim \max_{\substack{K \in \Th^\ell\\\partial K \cap \partial\omk \neq \emptyset}}  (p_K^2 h_K^{-1})^{(r-s)}  | \eta |_{H^{s}(E)} \lesssim \po^{2(r-s)} \ho^{s - r} | \eta |_{s,\partial\omk},  \label{inv:O}
\end{align}
where $\ho$  and $\po$  refer  to the minimum (resp. the maximum) of the restriction to $\partial\omk$ of the
local mesh size $\h$ (resp. the local polynomial degree function $\p$),
that is,
\begin{equation}\label{def:hpl}
\ho=\min_{x\in \partial\omk \cap \Gamma} \h(x)
\quad \text{and} \quad
\po=\max_{x\in \partial\omk \cap \Gamma} \p(x) .
\end{equation}

We write conventionally
\begin{equation}\label{quot:convention}
\frac{H\,p^2}{h} = \max_\ell \left \{ \frac{H_\ell\, \po^2}{\ho}\right \} .
\end{equation}

%%%%%%%%%%%%%%%%%
%%% Remark comentata, in caso che il referee lo chieda l'aggiungiamo dopo... ma per adesso va %%comentata
%%%%%%%
%\begin{remark} All the results proven in the following would also hold if we replaced $h_\ell$ and $p_\ell$ with as
%\[
%\hat h_\ell = \min_{K \in {\Th^\ell}_{|_{\partial\omk}}} h_K, \qquad \hat p_\ell = \max_{K \in {\Th^\ell}_{|_E}} p_K.
%\]
%However we remark that the final result would be the same up to a constant. In fact it is not difficult to prove that under our assumptions we have
%\[
%\max_\ell \left \{ \frac{H_\ell\, \po^2}{\ho}\right \}  \simeq \max_\ell \left \{ \frac{H_\ell\, \hat p_\ell^2}{\hat h_\ell}\right\}
%\]
%\end{remark}
%%%%%%%%%%%%%%%%%%%%%%%
%%%%%%%%%%%%%%%%%%%%%%%%%%%%%

The next two results  generalize  \cite[Lemma~3.2, 3.4 and
3.5]{BPS} and \cite[Lemma~3.2]{Bsubstr} to the $hp$-version.
The detailed proofs are reported in the Appendix \ref{app0}.
\begin{lemma} \label{lembsp2}
Let $\eta=(\eta^\ell)_{\ell=1}^N \in  \traceh$ and let $ \chi = (\chi^\ell)_{\ell=1}^N \in \Lineari$ be such that
$\chi^\ell(a)= \eta^\ell(a)$ at all vertices $a$ of $\omk$, for all $\omk\in \Ts$. Then
\begin{equation*}
\sum_{\omk\in \Ts}|\chi^\ell |_{H^{1/2}(\partial \omk)}^2 \lesssim
\left(1+\log{\left(\frac{H\,p^{2}}{h}\right)}\right)
 \sum_{\omk\in \Ts} |\eta^\ell |^2_{H^{1/2}(\partial
  \omk)}\;.
\end{equation*}
 \end{lemma}
\begin{lemma}\label{lembsp}
  Let $\xi \in \Tk$ such that $\xi(a) = 0$ at all  vertices
  $a$ of $\omk$. Let $\zeta_L\in H^{1/2}(\partial \omk)$ be linear on each subdomain edge of
$\partial \omk$. Then, it holds
  \begin{equation*}
    \sum_{E\subset \partial\omk} \| \xi \|_{H^{1/2}_{00}(E)}^2\lesssim
\left(1+\log{\left(\frac{\Ho\,\po^{2}}{\ho}\right)}\right)^2
\left | \xi + \zeta_L \right
    |_{H^{1/2}(\partial\omk)}^2 \; ,
      \end{equation*}
     where $\ho$ and $\po$ are defined in \eqref{def:hpl} and $\Ho$ is defined as in \eqref{def:hls}.
    \end{lemma}

\subsection{Norms on $\traceh$}
We now introduce a suitable norm on $\traceh$  that will suggest how to properly
construct the preconditioner.
The natural norm that we can define
for all $\eta=(\eta_\ell)_\ell\in \traceh$ is:
\begin{equation}\label{Snormh}
  \Shnorm{\eta}  =
  \inf_{
  \begin{array}{c}
       u\in  \Xh \\
      {u}_{|\Sigma}=\eta
  \end{array}} \normAh{u}\;,
\end{equation}
where the $\inf$ is taken over all $u\in \Xh$ that coincide with
$\eta$ along $\Sigma$. We recall that since on $\Gamma$ both $u$ and $\eta$ are double valued, the identity $\eta = u_{|_\Sigma}$ is to be intended as $\eta^\ell = u^\ell_{|_{\partial\omk}}$.
Although \eqref{Snormh} is the natural trace norm induced on $\Phi$ by the norm (\ref {def:normA}),  working with it might be difficult.

For this reason, we introduce another norm which will be easier to deal with and which, as we
will show below, is equivalent to \eqref{Snormh}. The structure of the
preconditioner proposed in this paper will be driven by this norm.
We define:\label{normDG}
\begin{equation}
\begin{aligned}
  \ShnormDG{\eta}^2 &= \sum_{\omk\in \Ts} |\eta |^2_{H^{1/2}(\partial
    \omk)} +  \alpha
\sum_{e\in \Eh}\| \p \,\h^{-1/2} \jump{\eta}\|_{L^{2}(e)}^{2}.
\end{aligned}
\end{equation}
The next result shows that the norms \eqref{Snormh} and \eqref{normDG} are indeed equivalent:
\begin{lemma}\label{equivT}
 The following norm equivalence holds:
\begin{equation*}
\begin{aligned}
 \Shnorm{\eta}  \lesssim \ShnormDG{\eta} \lesssim \Shnorm{\eta}
&& \forall \eta\in \traceh
\end{aligned}
\end{equation*}
\end{lemma}
\begin{proof}
We first prove that $\ShnormDG{\eta} \lesssim \Shnorm{\eta} $. Let $\eta=(\eta_{\ell})_{\ell=1,\dots,\Ndom}\in \traceh$ and let
  $u=(u_{\ell})_{\ell=1,\dots,\Ndom}$ such that
  ${u}_{|\skel}=\eta$ arbitrary. Thanks to the trace inequality, %\nref{stima12}
  we have
  $$
  |\eta_{\ell}|^2_{H^{1/2}(\partial \omk)} \lesssim
 |u_{\ell}|^2_{H^1(\omk)}\;,
   $$
   and so, summing over all
the subdomains $\omk \in \Ts$ we have
 $$
 \sum_{\omk\in \Ts} |\eta_{\ell}|^2_{H^{1/2}(\partial \omk)} \lesssim
 \sum_{\omk\in \Ts} |u_{\ell}|^2_{H^1(\omk)} =|u|_{1,\Ts}^{2}\;.
   $$
Adding now the term $\alpha \sum_{e\in \Eh}\| \p \h^{-1/2} \jump{\eta}\|_{L^2(e)}^{2}$
to both sides, and recalling the definition of the norms \eqref{def:normA}, \eqref{Snormh} and \eqref{normDG} we get the thesis thanks to the arbitrariness of $u$.

We now prove that
  $ \Shnorm{\eta}\lesssim \ShnormDG{\eta} $. Given
  $\eta=(\eta_{\ell})_{\ell=1,\dots,L}\in \traceh$
  let  $\check u_\ell \in X_h^\ell$ be the standard discrete harmonic lifting of $\eta^\ell$, for which the bound
$\snorm{\check u_\ell}_{H^1(\omk)} \lesssim
\snorm{\eta_\ell}_{H^{1/2}(\partial \omk)}$ holds
 (see e.g. \cite{BPS}) and let
$\check u = (\check u^\ell)_{\ell=1,\dots,L}$
Summing over all the subdomains $\omk$ and adding the term
$ \alpha  \sum_{e\in \Eh}\| \p \h^{-1/2} \jump{\eta}\|_{L^2(e)}^{2}$ we get
 \begin{equation*}\label{disatau}
\Shnorm{\eta} \leq \normAh{\check u}\lesssim \ShnormDG{\eta}.
\end{equation*}
\end{proof}
%%%%%%%%%%%%%%%%%%%%%%%%%%%%%%%%%%%%%%%%%%%%%%%%%%%%%%%%%%%%%%%%%%%%%%

%%%%%%%%%%%%%%%%%%%%%%%%%%%%%%%%%%%%%%%%%%%%%%%%%%%%%%%%%%%%%%%%%%%%%%
\section{Substructuring preconditioners} \label{sec:substr}
%{\rosso
In this section we present the construction and analysis of a substructing preconditioner for the Nitsche method \eqref{discrete_problem}-\eqref{def:nit}.	

The first step in the construction is to split the set of degrees of
freedom into {\em interior}
%{\rosa
degrees of freedom
%}
 (corresponding to basis functions identically vanishing on
the skeleton) and degrees of freedom associated to the skeleton
$\Gamma$ of the subdomain
partition.
%{\rosa
Then,%},
the idea  of the ``substructuring'' approach (see \cite{BPS}) consists in further distinguishing  two types among the degrees of freedom associated to $\Gamma$ :  {\em edge} degrees
of freedom and {\em vertex} degrees of freedom. Therefore,  any function $u \in \Xh$ can be split as the sum of three suitably defined components: $u = u^{0}+u_{\Gamma}=u^0 + u^E + u^V$.

We first show how to {\it eliminate} (or condensate) the interior degrees of freedom and introduce the discrete Steklov-Poincar\'e operator associated to \eqref{def:nit}, acting on functions living on the  skeleton of the subdomain partition. We then propose a preconditioner of substructuring type for the discrete Steklov-Poincar\'e operator and provide the convergence analysis.

\subsection{Discrete Steklov-Poincar\`e operator}
%{\rosso
 Following  \cite{Dryja.Galvis.Sarkis.2007, sarkis1},
 we now introduce a discrete harmonic lifting that allows for defining
 the discrete Steklov-Poincar\'e operator associated to
 \eqref{def:nit}.  We also show that such %{\rosa
 a
%}
 discrete
 Steklov-Poincar\'e operator defines a norm %{\rosa
that is
%}
equivalent
 to the one defined in \eqref{normDG}.\\
%}

Let $\Xho \subset \Xh$ be the subspace of
functions vanishing on the skeleton of the decomposition. Given any
discrete function  $w \in \Xh$, we can split it as the sum of an {\em interior} function $w^0 \in \Xho$
and a suitable discrete lifting %, performed subdomainwise
of its trace.
More precisely, following \cite{Dryja.Galvis.Sarkis.2007, sarkis1}, we split
$$
  w = w^0 + \Rh(w_{|\Sigma}), \qquad  w^0 \in \Xho,
$$
where, for $\eta \in \Phi_h$, $\Rh(\eta) \in \Xh$ denotes  the
unique element of $\Xh$ satisfying
\begin{equation}\label{def:Rh}
\begin{aligned}
&\Rh(\eta)_{|_\Sigma} = \eta, \quad
&&\calA_h(\Rh(\eta),v_h) = 0
&&\forall v_h \in \Xho.
    \end{aligned}
\end{equation}
The following proposition is easy to prove
%{\rosso
 (see \cite{Dryja.Galvis.Sarkis.2007, sarkis1}).
%}
%%
\begin{proposition}
For $\eta=(\eta^\ell) \in \Phi_h$, the following identity holds:
\begin{equation*}
R_h(\eta)_{|_{\omk}} = w_\ell^H +  w_\ell^0,
\end{equation*}
with $w_\ell^H \in X_h^\ell$ denoting the standard discrete harmonic lifting of $\eta^\ell$
\begin{equation*}
\begin{aligned}
&{w_\ell^H} = \eta^\ell \text{ on } \partial\omk, \quad
&&\int_\omk \nabla w_\ell^H \cdot \nabla v_h^\ell = 0
&&\forall v_h^\ell \in X_h^\ell \cap H^1_0(\omk),
\end{aligned}
\end{equation*}
and $w_\ell^0 \in X_h^\ell \cap H^1_0(\omk)$ being the solution of
\begin{equation*}
\begin{aligned}
&\int_{\omk} \nabla w_\ell^0 \cdot \nabla v_h^\ell = \int_{\partial\omk} \jump{\eta} \cdot \nabla v_h^\ell,
&& \forall v_h^\ell \in X_h^\ell \cap H^1_0(\omk).
\end{aligned}
\end{equation*}
\end{proposition}

The space $\Xh$ can be
split as direct sums of an interior and a
trace component, that is
\begin{equation*}
  \Xh = \Xho \oplus \Rh(\traceh) .
\end{equation*}
Using the above splitting, the definition of $\Rh(\cdot)$ and the definition of
 $\calA_h(\cdot,\cdot)$, it is not difficult to verify that,
 \begin{align*}
\calA_h(w,v) &= \calA_h(w^0,v^0) +
\calA_h(\Rh(w_{|_\Sigma}),\Rh(v_{|_\Sigma}))  \\
&= a(w^0,v^0) + s(w_{|_\Sigma},v_{|_\Sigma}), \qquad \qquad  \forall\, w,v \in
\Xh
\end{align*}
where the {\em discrete Steklov-Poincar\'e} operator  $s: \traceh
\times \traceh \to \RR$ is defined as
\begin{equation}\label{eqn:steklov}
s(\xi,\eta) =   \calA_h(\Rh(\xi),\Rh(\eta)) \qquad \forall\, \xi , \eta \in \traceh\;.
\end{equation}

We have the following result:

\begin{lemma}\label{lifting_ok}
Let $\Rh$ be the discrete harmonic lifting defined in \eqref{def:Rh}. Then,
\begin{equation*}
\normAh{\Rh(\eta)} \simeq \ShnormDG{ \eta} \qquad \forall\, \eta\in \traceh\;.
\end{equation*}
\end{lemma}
\begin{proof}
If we  show that
$\normAh{\Rh(\eta)} \simeq \Shnorm{ \eta}$,
then the thesis follows thanks to the equivalence of the norms shown in Lemma~\ref{equivT}. First, we prove that
$\normAh{\Rh(\eta)} \lesssim \Shnorm{ \eta}$;
let $\eta\in \traceh$,  then from the  definition of the  $\inf$,
we get that
\begin{equation*}
\exists u\in \Xh\,:\, u_{|\Sigma} = \eta \quad \mbox{ such that   }
\quad  \normAh{u} \leq 2 \Shnorm{\eta}\;.
\end{equation*}
Then, we can write  $\Rh(\eta) = u+v$  with $v\in\Xho$, and \nref{def:Rh} reads
\begin{equation*}
\calA_h(v,w) = -\calA_h(u,w)\qquad \forall w \in \Xho\;.
\end{equation*}

Setting $w=v \in \Xho$ in the above equation, leads to
      $$  \calA_h(v,v) = -\calA_h(u,v)\;. $$
 Then,  using the
 coercivity and continuity of
$\calA_h(\cdot,\cdot)$ in the $\normAh{\cdot}$-norm
we find
\begin{equation*}
 \normAh{v}^{2}
\lesssim \calA_h(v,v)=|\calA_h(u,v)|
 \lesssim \normAh{u}\normAh{v}\;.
%\|u\|_{\calA}\|v\|_{\calA} \;.
\end{equation*}
Hence, $\normAh{v} \lesssim \normAh{u}$, and so this bound together with the triangle inequality gives
\begin{equation*}
\normAh{\Rh(\eta)}\leq  \normAh{u}+\normAh{v} \lesssim
\normAh{u} \lesssim
\Shnorm{\eta}\;.
\end{equation*}
%    \text{ follows from
%   elliptic regularity for elliptic functions}
The other inequality
$\Shnorm{ \eta} \lesssim  \normAh{\Rh(\eta)}$
follows from the trace theorem.
\end{proof}
 From the above result, the following result for the \emph{discrete Steklov-Poincar\`e} operator follows easily.
\begin{corollary}\label{equiv_norm_Ttau}
For all   $\xi \in \traceh$, it holds
\begin{equation*}
 s(\xi,\xi) \simeq \ShnormDG{\xi}^2.
\end{equation*}
\end{corollary}
\begin{proof}
Let $\xi\in \traceh$ then from the  definition of $s(\cdot,\cdot)$, the continuity and coercivity
of $\calA_h(\cdot,\cdot)$  and applying Lemma~\ref{lifting_ok} we have
\begin{equation*}
s(\xi,\xi) = \calA_h(\Rh(\xi),\Rh(\xi)) \simeq \normAh{\Rh(\xi)}^2\simeq \ShnormDG{\xi}^2.
\end{equation*}
\end{proof}

%

%%%%%%%%%%%%%%%%%%%%%%%%%%%%%%%%%%%%%%%%%%%%%%%%%%%%%%%%%%%%%%%%%%%%%%
\subsection{The preconditioner}\label{sec:5.2}

 Following the approach introduced in \cite{BPS}, we now present the
 construction of a preconditioner for the discrete Steklov-Poincar\'e
 operator given by $s(\cdot,\cdot)$.
We split the space of skeleton functions $\traceh$ as the sum of {\em vertex} and {\em edge} functions.
We start by observing that
$\Lineari \subset \traceh$.
%, which yields $\traceh^V \subset \traceh$.
%
We then  introduce the space of {\em edge} functions
$\traceh^E \subset \traceh$ defined by
\begin{equation*}
  \traceh^E = \{ \eta \in \traceh,~ \eta_\ell(A)=0~
  \text{ ~at all vertex $A$ of} ~\omk \quad \forall\, \omk\in \Ts \}
\end{equation*}
and we immediately get
\begin{equation} \label{eq:11}
  \traceh = \Lineari \oplus \traceh^E .
\end{equation}
%
%%%%%%%%%%%%%%%%%%%%%%%%%%%%%%%%%%%%%%%%%%%%%%%%%%%%%%%%%%%%%%%%%%%%%%
The preconditioner $\shat(\cdot,\cdot)$ that we consider is built
by introducing  bilinear forms
\begin{equation*}
\begin{aligned}
& \sE:\traceh^E \times \traceh^E \lor \RR
&& \qquad  \sV:\Lineari \times \Lineari \lor \RR
\end{aligned}
\end{equation*}
acting respectively on edge and vertex functions, satisfying
\begin{align}
\sE(\eta^E,\eta^E) &
\simeq \sum_{\omk\in \Ts}\sum_{E\subset \partial\omk}\norm{\eta^E}^2_{H^{1/2}_{00}(E)}
&& \forall\, \eta^{E}\in \traceh^E, \label{eq:sE}\\
\sV(\eta^V,\eta^V) &\simeq \sum_{\omk\in \Ts}
\left |\eta^V\right |^2_{H^{1/2}(\partial \omk)}
&& \forall\, \eta^{V} \in \Lineari,
%\traceh^V,
\label{eq:sV}
\end{align}%
and we define $\shat: \traceh\times \traceh \lor \RR$ as
\begin{equation}
  \label{shat}
  \shat(\eta,\xi) =  \sE(\eta^E,\xi^E) + \sV(\eta^V,\xi^V) + q(\eta,\xi),
\end{equation}
where \label{normDG}
\begin{align}
q(\eta,\eta) &= \alpha \sum_{e\in \Eh} \| \p \,\h^{-1/2}\jump{\eta}\|_{L^2(e)}^2
&&\forall\, \eta\in \traceh\;. \label{eq:jEV}
\end{align}%

Finally, we  can state the main theorem of the paper.
%\smallskip
%The following theorem holds:
\begin{theorem}\label{precond}
Let $s(\cdot,\cdot)$ and $\shat(\cdot,\cdot)$ be the bilinear forms defined in \eqref{eqn:steklov} and \eqref{shat}, respectively. Then, we have:
  \begin{equation*}
  \faclog^{-2}  \shat(\eta,\eta) \lesssim s(\eta,\eta) \lesssim  \shat(\eta,\eta)  \qquad\forall\, \eta \in \traceh\;.
  \end{equation*}
% Moreover, if the decomposition is geometrically conforming then
%  \begin{equation}\label{log2}
%   s(\eta,\eta) \lesssim \shat(\eta,\eta) \lesssim \faclog^2 s(\eta,\eta).
%   \end{equation}
\end{theorem}

%{\rosso
The proof of Theorem \ref{precond} follows the analogous proofs given
in \cite{BPS,Bsubstr} for conforming finite element approximation.
%}
We give it here for completeness.

\begin{proof}
We start proving that
$s(\eta,\eta) \lesssim \shat(\eta,\eta)$. Let  $\eta \in \traceh$, then,
$\eta= \eta^V + \eta^E$ with $\eta^E\in \traceh^E$ and $\eta^V\in
\Lineari$.
By using Corollary~\ref{equiv_norm_Ttau}, as well as the properties \nref{eq:sE}-\nref{eq:sV} of the
edge and vertex bilinear forms,  and \nref{eq:jEV} of $q(\cdot,\cdot)$, we get
\begin{align*}
s(\eta,\eta) & \lesssim
\ShnormDG{\eta}^2 = \sum_{\omk\in\Ts} \snorm{\eta^E +
\eta^V}^2_{1/2,\partial\Omega_\ell} + \alpha \sum_{e\in \Eh} \| \p\,
\h^{-1/2} \jump{\eta}\|_{L^2(e)}^2  \\
& \lesssim   \sum_{\omk\in\Ts} \snorm{\eta^E}^2_{1/2,\partial\Omega_\ell} +
\sum_{\omk\in\Ts} \snorm{\eta^V}^2_{1/2,\partial\Omega_\ell} +
q(\eta,\eta)  \\
& \lesssim   \sE(\eta^E,\eta^E)  +   \sV(\eta^V,\eta^V)  +
q(\eta,\eta),
% =   \shat(\eta,\eta)%
 \end{align*}
and hence
\begin{equation*}
s(\eta,\eta) \lesssim   \shat(\eta,\eta) \qquad\forall\, \eta \in \traceh\;.
\end{equation*}

\

We next prove the lower bound. We shall show that
 \begin{equation}
        \label{boundonb}
        \shat(\eta,\eta) \lesssim \faclog^2 s(\eta,\eta)  \qquad \forall\,  \eta \in \traceh\;.
      \end{equation}

For $\eta \in \traceh$, we have $\eta= \eta^V + \eta^E$ with $\eta^E\in  \traceh^E$ and $\eta^V\in
\Lineari$. Then, from the definition of $\shat(\cdot,\cdot)$ we have
\begin{align*}
 \shat(\eta,\eta)& = \sE(\eta^E,\eta^E) + \sV(\eta^V,\eta^V) + q(\eta,\eta) \\
   & \simeq   \sum_{\omk\in \Ts}\sum_{E\subset \partial\omk} \norm{\eta^E}^2_{H^{1/2}_{00}(E)}
 + \sum_{\omk\in \Ts}\left |\eta^V\right
 |^2_{H^{1/2}(\partial \omk)} +  \alpha \sum_{e\in \Eh} \| \p \,\h^{-1/2} \jump{\eta}\|_{L^2(e)}^2.
\end{align*}

Appling Lemma~\ref{lembsp}  with $\chi=\eta^{E}$ and $\zeta_{L}=\eta^{V}$, we obtain
\begin{equation*}
     \sum_{\omk\in \Ts}\sum_{E\subset \partial\omk}
  \norm{ \eta^E}^2_{H^{1/2}_{00}(E)} \lesssim
     \sum_{\omk\in \Ts}\faclogi^2 \snorm{ \eta}^2_{H^{1/2}(\omk)}, %\partial\omk)}
  \end{equation*}
  that is
\begin{equation*}
\sE(\eta^E,\eta^E) \lesssim \faclog^2  \sum_{\omk\in
  \Ts} \snorm{\eta}^2_{H^{1/2}(\partial \omk)}   .
\end{equation*}

To bound $\sV(\eta^V,\eta^V)$, we apply Lemma~\ref{lembsp2}  with $\chi^{\ell}=\eta^{V}$ and $\eta^{\ell}=\eta$, and we get
\begin{align*}
\sV(\eta^V,\eta^V) \lesssim \sum_{\omk\in \Ts}
\snorm{\eta^V}^2_{H^{1/2}(\partial \omk)} \lesssim \faclog \sum_{\omk\in
  \Ts} \snorm{\eta}^2_{H^{1/2}(\partial \omk)},
\end{align*}
and hence
\begin{equation*}
\sE(\eta^E,\eta^E)  + \sV(\eta^V,\eta^V)  \lesssim \faclog^2
\sum_{\omk\in\Ts} \snorm{\eta}^2_{H^{1/2}(\partial \omk)}  .
\end{equation*}
Adding now the term $ \alpha \sum_{e\in \Eh}\| \p \h^{-1/2} \jump{\eta}\|_{L^2(e)}^{2}$
to both sides and recalling the definition of $q(\cdot,\cdot)$ we have:
\begin{align*}
 \shat(\eta,\eta)& =\sE(\eta^E,\eta^E)   + \sV(\eta^V,\eta^V) +q(\eta,\eta) && \\
& \lesssim \faclog^2 \left(
\sum_{\omk\in
  \Ts} \snorm{\eta}^2_{H^{1/2}(\partial \omk)} +  \alpha \sum_{e\in \Eh}\| \p
\h^{-1/2} \jump{\eta}\|_{L^2(e)}^{2} \right)  && \\
& =  \faclog^2  \ShnormDG{\eta}^2. &&
\end{align*}

Finally, using the equivalence norm given in Corollary~\ref{equiv_norm_Ttau}, we reach \eqref{boundonb} and the proof of the Theorem is completed.

  \end{proof}

%{\rosso
As a direct consequence of Theorem \ref{precond} we obtain the
following estimate for the condition number of the preconditioned
Schur complement.
%}
\begin{corollary}\label{cond}
  Let $\bS$ and $\bP$ be the matrix representation of the bilinear forms $s(\cdot,\cdot)$ and $\shat(\cdot,\cdot)$, respectively. Then, the condition number of $\bP ^{-1} \bS$,   $ \kappa (\bP ^{-1} \bS)$, satisfies
  %for $\theta = 1$,
%  it holds
\begin{equation}
  \kappa (\bP ^{-1} \bS) \lesssim \faclog^2.
\end{equation}
\end{corollary}

%{\rosso
Unfortunately,  the splitting (\ref{eq:11}) of $\Phi_h$ is not
orthogonal with respect to the $\hat{s}(\cdot,\cdot)$-inner product given
in \eqref{shat}, and therefore the preconditioner based on
$\hat{s}(\cdot,\cdot)$ is not block diagonal, in contrast to what
happens in the full conforming case. Furthermore the off-diagonal
blocks in the preconditioner cannot be dropped without loosing the
quasi-optimality.  The reason is the presence of the $q(\cdot,\cdot)$
bilinear form in the definition \eqref{shat}, and the fact that the
two components in the splitting (\ref{eq:11}) of $\Phi_h$ scale
differently in the semi-norm that $q(\cdot,\cdot)$ defines.
%}\\
In fact, it is possible to show that, if for some constant $\kappa(h)$, it holds
\begin{equation}\label{false-dis}
\begin{aligned}
&\| \eta^V \|_{\Phi_h,*}^2\leq \kappa(h)\| \eta \|_{\Phi_h,*}^2
&& \forall \eta = \eta^V+\eta^E \in \Phi_h,
\end{aligned}
\end{equation}
then such $\kappa(h)$ must verify
$\kappa(h)\gtrsim H/h$, which implies that, if we were to use a fully block diagonal preconditioner based on the splitting (\ref{eq:11}) of $\Phi_h$  an estimate of the form (\ref{boundonb}) would no be longer true.
In order to show this, consider linear finite elements on quasi uniform meshes with meshsize $h$ in all subdomains, and let $\eta = (\eta^\ell)_{\ell}$ be the function identically vanishing in all subdomains but one, say $\Omega_k$, and let $\eta^k$ be equal to $1$ in a single vertex of $\Omega_k$ and zero at all other nodes. With this definition, we have $|\jump{\eta}| = |\eta^k|$ on $\partial \Omega_k$ and $\jump{\eta}=0$ on $\Sigma\setminus\partial\Omega^k$. Then, by a direct calculation, and recalling the definition of the semi-norm $|\cdot|_{\ast,\Eh}$ in \eqref{eq:seminorms}, we easily see that
%{\rosso
$$|\eta|^{2}_{\ast,\Eh}\simeq 1, \quad \mbox{ but }\qquad |\eta^{V}|^{2}_{\ast,\Eh}\simeq \frac{H_k}{h_k} $$
%}
or equivalently
\begin{equation}\label{controesempio1}
  q(\eta^V,\eta^V) \simeq \frac {H_k}{h_k}\;, \qquad  q(\eta,\eta)\simeq 1.
\end{equation}
Therefore the energy of {\it coarse interpolant} $\eta^{V}$ exceeds that of $\eta$ by a factor of $H_k/h_k$.
Hence, bounding  $\eta^{V}$ alone in the $\| \cdot\|_{\Phi_h,*}$-norm would result in an estimate of the type \eqref{false-dis}
\begin{equation}\label{controesempio2}
q(\eta^V,\eta^V) \lesssim  \| \eta^V \|_{\Phi_h,*}^2 \lesssim   \kappa(h) q(\eta,\eta),
\end{equation}
which in view of (\ref{controesempio1}) would imply
\[ \kappa(h) \gtrsim \frac{H_k}{h_k}. \]

\begin{remark}
%{\rosso
%Due to the presence of the bilinear form $q(\cdot,\cdot)$  in the
%definition of the norm \eqref{normDG} (equivalent to that induced by $s(\cdot,\cdot)$) and in the definition of \nref{shat} of
%$\shat(\cdot,\cdot)$, the preconditioner built by using
%the matrix counterpart of $\shat(\cdot,\cdot)$ loses the
%block diagonal structure that it typically has in the conforming case.
 We point out that the lack of the
block-diagonal structure of the
preconditioner associated to  $\shat(\cdot,\cdot)$ defined in \nref{shat},
will not affect its computational efficiency, see
Section~\ref{sec:expes}.
%}
\end{remark}

%%%%%%%%%%%%%%%%%%%%%%%%%%%%%%%%%%%%%%%%%%%%%%%%%%%%%%%%%%%%%%%%%%%%%%%

%%%%%%%%%%%%%%%%%%%%%%%%%%%%%%%%%%%%%%%%%%%%%%%%%%%%%%%%%%%%%%%%%%%%%%
\section{Realizing the preconditioner}\label{sec:variants}

%\subsection{Matrix form}\label{sec:matrixform}
%\COMMENT
%{\rosso Per riduzioni lasciamo ai referee suggerirlo...}

We start by deriving the matrix form of the discrete
Steklov-Poincar\'e operator $s(\cdot,\cdot)$ defined in \nref{eqn:steklov}.
We choose a Lagrangian nodal basis for the discrete space $X_{h}$, and we take care of
 numbering interior degrees of freedom first (grouped subdomain-wise),
 then edge degrees of freedom (grouped edge by edge and in such a way
 that the degrees of freedom corresponding to the common edge of two
 adjacent subdomains are ordered consecutively), and finally the
 degrees of freedom corresponding to the vertices of the subdomains.
%\footnote{\blu aggiunto che bisogna numerare i gradi di libert\`a in modo opportuno}
%
%  \nref{discrete_problem} can be reduced to the following linear system of equations:
%\begin{equation*}
%  \label{linearsystem}
%  \bA \bu = \bF.
%\end{equation*}
We let $\iin_{\ii}$, $\iin_{\ie}$ and $\iin_{\iv}$ be the number of
interior, edge and vertex degrees of freedom, respectively, and set
$\iin=\iin_{\ie}+\iin_{\iv}$
%\footnote{\blu modificata la definizione di $\iin$ che ora \`e il
%numero di gradi di libert\`a sullo scheletro, che son quelli su cui
%si fa la maggior parte del lavoro}.
Problem \nref{discrete_problem} is then reduced to looking for a vector  ${\bu}\in \re^{\iin_\ii+\iin}$
with
%{\rosa
$\bu =(\bu_{\ii},\bu_{\ie}, \bu_{\iv})$
%}
solution to a linear system of the following form
\begin{equation*}
\begin{pmatrix}
    \bA_{\ii \ii}       & \bA_{\ii \ie}   & \bA_{\ii \iv}    \\
    \bA_{\ii \ie}^T & \bA_{\ie \ie}  & \bA_{\ie \iv}    \\
    \bA_{\ii \iv}^T & \bA_{\ie \iv}^T & \bA_{\iv \iv}
\end{pmatrix}
\begin{pmatrix}
\bu_{\ii} \\ \bu_{\ie} \\\bu_{\iv}
\end{pmatrix} =
\begin{pmatrix}
\bF_{\ii} \\ \bF_{\ie} \\\bF_{\iv}
\end{pmatrix}.
\end{equation*}

Here, $\bu_{\ii}\in \re^{\iin_{\ii}}$ (resp. $\bF_{\ii} \in \re^{\iin_{\ii}}$) represents the
unknown (resp. the right hand side) component associated to interior nodes.
Analogously, $\bu_{\ie}, \bF_{\ie} \in \re^{\iin_{\ie}}$ and
$\bu_{\iv}, \bF_{\iv} \in \re^{\iin_{\iv}}$ are associated to edge and
vertex nodes, respectively. We recall that for each vertex we have one
degree of freedom for each of the subdomains sharing it. For each
macro edge $E$, we will have two sets of nodes (some of them possibly
physically coinciding) corresponding to the degrees of freedom of $\Phi_h^{\ell^+}(E)$ and of $\Phi_h^{\ell^-}(E)$.

\

As usual, we start by eliminating the interior degrees of freedom, to obtain
the Schur complement system
\begin{equation*}
\bS
\begin{pmatrix}
\bu_{\ie} \\ \bu_{\iv}
\end{pmatrix}
=\bg,
\end{equation*}
with
\begin{equation*}\label{eq:defS}
\begin{aligned}
&\bS =
\begin{pmatrix}
\bA_{\ie \ie} - \bA_{\ii \ie}^T \bA_{\ii \ii}^{-1} \bA_{\ii \ie}
& \bA_{\ie \iv}-\bA_{\ii \ie}^T \bA_{\ii \ii}^{-1} \bA_{\ii \iv}\\
 \bA_{\ie \iv}^T-\bA_{\ii \iv}^T \bA_{\ii \ii}^{-1} \bA_{\ii \ie}
 & \bA_{\iv \iv}-\bA_{\ii \iv}^T \bA_{\ii \ii}^{-1} \bA_{\ii \iv}
\end{pmatrix},
&&
\bg =
\begin{pmatrix}
\bF_E - \bA_{\ii \ie}^T \bA_{\ii \ii}^{-1} \bF_{\ii} \\
\bF_V - \bA_{\ii \iv}^T \bA_{\ii \ii}^{-1} \bF_{\ii}
\end{pmatrix}.
\end{aligned}
\end{equation*}
The Schur complement $\bS$ represents
the matrix form of the Steklov-Poincar\'e operator $s(\cdot, \cdot)$. Remark that in practice we do not need to actually assemble $\bS$ but only to be able to compute its action on vectors.

\

In order to implement the preconditioner introduced in the previous
section we need to represent algebraically the splitting of the trace space given by
\eqref{eq:11}.
%{\rosa
 As defined
in \nref{eq:Lineari},
%},
we consider the space $\Lineari$ of functions
%{\rosa
that are linear
%}
on each subdomain edge,  and introduce the
matrix representation of the injection of $\Lineari$ into $\Phi_h$.
%\footnote{\blu Qui lo spazio delle lineari sugli edges $e$ non va bene, perch\'e nella teoria consideriamo anche elementi che non sono lineari, secondo me \`e pi\`u corretto scrivere cosi}
More precisely, we let
%\rosa
$\bXi = \{ \bx_{i}, \, i=1,\ldots,\iin_{\ie},
\iin_{\ie}+1,\ldots, \iin_{\ie}+\iin_{\iv}\}$
%}
be the set of edge and
vertex %\st{nodes} {\rosa
degrees of freedom.
%}
%\footnote{{\rosa Secondo me nella versione precedente un po' copmlicato da leggere per chi viene dalla comunita' DG. Ho provato a semplificare. La versione precedente e' commentata nel tex.}}
%{\rosa
For any vertex degree of freedom
$\bx_j$, $j=\iin_{\ie}+1,\ldots,\iin_{\ie}+\iin_{\iv}$ , let $\varphi_{j}(\cdot)$ be the piecewise polynomial that is linear on each subdomain edge and that satisfies
\begin{equation*}
\begin{aligned}
\varphi_{j}(\bx_k)=\delta_{jk}
&& j,k=\iin_{\ie}+1,\ldots,\iin_{\ie}+\iin_{\iv}.
\end{aligned}
\end{equation*}
%}

%where each physical node is repeated with the
%multiplicity corresponding to the number of degrees of freedom
%attached to it, so that each $\bx_j \in \bXi$ is attached to a unique
%degree of freedom $\phi_h^\ell(\bx_j)$, that we will pinpoint by
%saying $\bx_j \in \Omega^\ell$.
%For $\bx_j \in \Omega_\ell$ and $\phi_h = (\phi_h^\ell)_{\ell=1}^N \in
%\Phi_h$
%we will write $\phi_h(\bx_j)$ for $\phi_h^\ell(\bx_j)$.
%Let now $\{\varphi_{j}, \,
%j=\iin_{\ie}+1,\ldots,\iin_{\ie}+\iin_{\iv}\}$ be the nodal basis for
%$\Lineari$, that is $\varphi_j = (\phi_j^\ell)_{\ell=1}^N$ with
%$\varphi_j^\ell$ linear on each edge of $\Omega_\ell$, and $\varphi_j$
%verifies
%\begin{equation*}
%\varphi_{j}(\bx_i)=
% \delta_{ij} \quad
%j = \iin_{\ie}+1,\ldots,\iin_{\ie}+\iin_{\iv}, \,\bx_i\in \bXi .
%\end{equation*}
% \begin{equation*}
% \varphi_{j}(\bx)=
% \left\{
% \begin{aligned}
% & 1 &&\textrm{if $\bx=\bx_{j}$}\\
% & 0 &&\textrm{if $\bx\neq \bx_{j}$}
% \end{aligned}
% \right.
% \qquad
% j = \iin_{\ie}+1,\ldots,\iin_{\ie}+\iin_{\iv}.
% \end{equation*}
%
The matrix $\bR^{T}\in \re^{\iin\times \iin_{\iv}}$ realizing the
linear interpolation of vertex values % on edge nodes
 is then defined as
\begin{equation*}
\begin{aligned}
&\bR^{T}(i,j-\iin_{\ie}+1)=\varphi_{j}(\bx_{i}),
&& i=1,\ldots,\iin,
&& j=\iin_{\ie}+1,\ldots,\iin_{\ie}+\iin_{\iv}.
\end{aligned}
\end{equation*}
Next, we define a square matrix $\widetilde{\bR}^{T}\in \re^{\iin \times \iin}$ as
\begin{equation*}
 \widetilde{\bR}^{T} = \left(
 \begin{aligned}
& \begin{pmatrix}
 \bI_{\ie} \\ {\bf 0}
 \end{pmatrix}
 && \bR^{T}
 \end{aligned}
 \right),
 \end{equation*}
$\bI_{\ie} \in \re^{\iin_{\ie}\times \iin_{\ie}}$ being the identity matrix.
Let now $\widetilde{\bS}$ be the
matrix obtained after applying the change of basis corresponding
to switching from the standard nodal basis to the  basis related
to the splitting \nref{eq:11}, that is
\begin{equation}\label{eq:Smodified}
\widetilde{\bS} = \widetilde{\bR} \bS \widetilde{\bR}^T =
     \left  (
  \begin{array}{cc}
    \widetilde{\bS}_{\ie \ie} & \widetilde{\bS}_{\iv \ie}\\
     \widetilde{\bS}_{\iv \ie}^T & \widetilde{\bS}_{\iv \iv}
  \end{array}\right ) .
\end{equation}
%
%Matrix $\widetilde{\bS}$ is the new Schur complement matrix, after
%applying the change of basis
%
%corresponding to the splitting \nref{Tcoarse}-\nref{eq:11}.
%
Our problem is then reduced to the solution of a transformed Schur complement system
 \begin{equation}\label{eq:Schur_system_modified}
\widetilde{\bS} \, \widetilde{\bu}= \widetilde{\bg},
\end{equation}
where $\widetilde{\bu} = \widetilde{\bR}^{-T} \bu$ and
$\widetilde{\bg}= \widetilde{\bR} \bg$.
%\footnote{\blu Corrette le espressioni di $\widetilde{\bS}$, $\widetilde{\bu}$ and $\widetilde{\bg}$.}\\

\

\noindent
{\it The preconditioner $\bP$}.
The preconditioner $\bP$ that we propose is obtained as matrix
counterpart of \nref{shat}.
In the literature it is possible to find different ways to build bilinear forms
%{\rosa
$\hat s^E(\cdot,\cdot)$, $\hat s^V(\cdot,\cdot)$
%}
that satisfy \nref{eq:sE} and \nref{eq:sV}, %{\rosa
respectively.
%}.
The choice that we make here for defining %{\rosa
$\hat s^E(\cdot,\cdot)$ %}
 is the one proposed in \cite{BPS} and it is based on an
equivalence result for the $H^{1/2}_{00}$ norm.
%{\rosso
We revise now its construction.
Let
%{\rosa
$l_0(\cdot)$
%}
denote the discrete operator defined on $\Phi^0_\ell(E)$ associated to the finite-dimensional approximation of $-\partial^2/\partial s^2$ on $E$. It is defined by:
\begin{equation}\label{eqH1200}
\langle l_0 \varphi,\phi\rangle_E = (\varphi^\prime,\phi^\prime)_E\qquad
%{\rosa
\forall
%}
 \phi \in \Phi^0_\ell(E),
\end{equation}
where the prime
%{\rosa
superscript
%}
refers, as usual, to the derivative $\partial/\partial s$ with
respect to the arc length $s$ on $E$.
%{\rosa
Notice that,
%}
since
%{\rosa
$l_0(\cdot)$
%}
is symmetric and positive definite, its square root can be defined.
Furthermore, %}
it can be shown that
$$\norm{\varphi}_{H^{1/2}_{00}(E)}\simeq
 (l_0^{1/2}\varphi, \varphi)^{1/2}_{E},$$
see \cite{BPS}.
Then, we define
\begin{equation}\label{sEprec}
\sE(\eta^E,\xi^E) = \sum_{\omk\in
  \Ts}\sum_{E\subset \partial\omk}  (l_0^{1/2}\eta^E,\xi^E)_{E}\; \qquad \forall\, \eta^E, \xi^E \in \Phi^0_\ell(E).
\end{equation}
 For $\eta^E\in \Phi^0_\ell(E)$ we denote by $\bbeta^E$ its vector
 representation.
Then, it can be verified that, for each subdomain edge
$E\subset \partial \omk$, we have (see \cite{Bjorstad:1986:IMS}
pag. 1110 and \cite{Dryja:1981:CMM})
\begin{equation*}
 (l_0^{1/2}\eta^E, \eta^E)_{E}
={\bbeta^E}^T  \widehat \bK_E\bbeta^E
\end{equation*}
where $ \widehat \bK_E=  \bM_E^{1/2} ( \bM_E^{-1/2}  \bR_E
   \bM_E^{-1/2})^{1/2}  \bM_E^{1/2}$, and where $ \bM_E$ and $ \bR_E$ are the 
   mass and stiffness matrices associated to the discretization
of the operator $-\textrm{d}^2/\textrm{ds}^2$ (in $\Phi^0_\ell(E)$) with
homogeneous Dirichlet boundary conditions at the extrema $a$ and $b$
of $E$.

%{\rosa
Observe, that for each macro edge $E$ shared by the subdomains
$\Omega_{\ell^+}$ and $\Omega_{\ell^-}$,
$ \widehat \bK_E$
is a two by two block diagonal matrix of the form
\begin{equation*}
 \widehat \bK_E =
\begin{pmatrix}
\widehat \bK_E^{+} & {\bf 0}\\
{\bf 0} & \widehat \bK_E^{-} \\
\end{pmatrix},
\end{equation*}
%}
where $\widehat \bK_E^\pm $
are the contributions  from the subdomains $\Omega_{\ell^\pm }$  sharing the macro-edge $E$.
%}
% we will have two
%contributions to \nref{sEprec}  and  correspondingly two matrices ${\rosa \widehat \bK_E^{+}}$ and ${\rosa \widehat \bK_E^{-}}$.}
%
%
%Since $\tilde R_E$ is definite positive, its square root is well
%defined.
%
As far as the vertex bilinear form $\hat s^V(\cdot,\cdot)$ is concerned, we choose:
\begin{equation}\label{sv}
\hat s^V(\eta^V,\eta^V) =  \sum_{\omk\in \Ts}\int_{\omk}
\nabla({\cal H}_h^\ell\eta^\ell) \cdot \nabla({\cal
  H}_h^\ell\eta^\ell)
%{\rosa
\dd{x},
%}
\end{equation}
where $\cal H(\cdot)$ denotes the standard discrete harmonic lifting \cite{BPS,Xu.Zou}.
We observe that if the $\omk$'s are rectangles, for $\eta\in\Lineari$ we
have that ${\cal H}_h^\ell\eta^\ell$ is the $\mathbb{Q}^{1}(\omk)$ polynomial that coincides with $\eta^\ell$
at the four vertices of $\omk$. Computing $\hat s^V(\eta^V,\xi^V)$ for
$\eta^V,\xi^V\in \Lineari$ is therefore easy, since it is reduced to compute the local (associated to $\omk$) stiffness matrix for $\mathbb{Q}^{1}(\omk)$ polynomials.

\begin{remark}
A similar construction also holds for quadrilaterals which are affine
images of the unit square, and for triangular domains.
In fact, if $\omk$ is a triangle then for $\eta\in\Lineari$ we
have that ${\cal H}_h^\ell\eta^\ell$ is the $\mathbb{P}^{1}(\omk)$ function coinciding with $\eta^\ell$
at the three vertices of $\omk$. If $\omk$ is the affine image of the unit square, we work by using the harmonic lifting on the reference element.
\end{remark}
The preconditioner $\bP $ can then be written as:
\begin{equation}\label{P2}
\bP  \!=\!\! \left(
 \begin{array}{ccccc}
\bK_{E_1} & 0 & 0 &0&0\\
0 & \bK_{E_2} & 0 &0&0\\
0 & 0& \ddots &0&0\\
0 & 0 &  0&\bK_{E_M}&0 \\
0&0 & 0 & 0& \bP_{\iv \iv}
 \end{array}
\right)  \!+  %{\rosa
\widetilde{\bQ}
%}
\,,
\end{equation}
where for each macro edge $E_i$,
\[\bK_{E_i}=\begin{pmatrix}
(\widehat  \bK_{E_i}^{+})^{1/2} & 0 \\
0 & (\widehat  \bK_{E_i}^{-})^{1/2}
\end{pmatrix}.
\]
%{\rosa
 In \eqref{P2} $\bP_{\iv\iv}$ is defined as
 the matrix counterpart of \nref{sv}
whereas $\widetilde{\bQ}= \widetilde{\bR} \bQ \widetilde{\bR}^T$ and
$$\bQ = \left(
 \begin{array}{cccc}
\bQ_{E_1} & 0 & 0 & \bQ_{E_1 V} \\
0 & \bQ_{E_2} & 0 & \bQ_{E_2 V} \\
0 & 0 & \ddots & \vdots \\
\bQ_{E_1 V}^T & \bQ_{E_2 V}^T
& \cdots & \bQ_{\iv \iv},
 \end{array}
\right)
$$
is the matrix counterpart of \nref{eq:jEV}.
%}
Remark that, due to the structure of the off diagonal blocks of $\bQ$, $\bP$ is low-rank
perturbation of an invertible block diagonal matrix.
The action of $\bP^{-1}$ can therefore be easily computed,  see
e.g. \cite{Demmel_1997} sec.2.7.4, p. 83.

\

{\it The preconditioner ${\bP}_\star$}.
For comparison we introduce a preconditioner ${\bP}_\star$ with the
same block structure of $\bP$ but with the elements of the non-zero
blocks coinciding with the corresponding elements of
$\widetilde{\bS}$. We expect this preconditioner to be the best that
can be done within the block structure that we want our preconditioner
to have.
In order to do so, we replace the $\widetilde{\bS}_{\ie\ie}$ component
of $\widetilde{\bS}$ with
the matrix obtained by dropping
all couplings between the degrees of freedom corresponding to nodes
belonging to different macro
edges, and use the resulting matrix as preconditioner.
More precisely, for any subdomain edge $E_{k}$ of the subdomain partition, $k=1,\ldots,M$, let
${\bJ}_{k} \in \re^{\iin_{\ie}\times\iin_{\ie}}$ be the diagonal matrix that extract only
the edge degrees of freedom belonging to the macro edge $E_{k}$,
i.e.,
\begin{equation*}
{\bJ}_{k}(i,j)=
\left\{
\begin{aligned}
&1 &&\textrm{if  $i=j$ and $\bx_{i} \in E_{k}$}\\
& 0 &&\textrm{otherwise}
\end{aligned}
\right.
\quad
i,j=1,\dots,\iin_{\ie}.
\end{equation*}
Then, we define
\begin{equation*}
\widetilde{\bP}_{\ie \ie} =\sum_{k=1}^{m} {\bJ}_{k}^{T} \widetilde{\bS}_{\ie \ie} {\bJ}_{k}
\end{equation*}
This provides our preconditioner
\begin{equation}\label{eq:defP1}
\bP_{\star}
  = \left (\begin{array}{cc}
       \widetilde{\bP}_{\ie\ie} & \widetilde{\bS}_{\ie \iv}  \\
       \widetilde{\bS}_{\ie \iv}^T & \widetilde{\bS}_{\iv \iv}
     \end{array}\right ) .
\end{equation}
Building this preconditioner implies the need of assembling at least
part of the Schur complement; this is quite expensive and therefore
this preconditioner is not feasible in practical applications.
%\COMMENT{{\rosa Spostato remark qui: usa $\widetilde{\bP}_{\ie\ie}$ che prima non era stato definito}}
\begin{remark}\label{rem:defP1diag}
Note that we cannot drop the coupling between edge and vertex
points, i.e. we cannot eliminate the off-diagonal blocks
$\bQ_{E_i V},\bQ_{E_i V}^T$. Indeed,
as already pointed out at the end of Section \ref{sec:5.2},
with the splitting (\ref{eq:11}) of $\Phi_h$ it is not
possible to design a block diagonal preconditioner without losing
quasi-optimality. In Section~\ref{sec:expes} we will present some
computations that show that the preconditioner \label{normDG}
\begin{equation}\label{eq:defP1diag}
\bP_{D}
  = \left (\begin{array}{cc}
       \widetilde{\bP}_{\ie\ie} & 0\\
    0 & \widetilde{\bS}_{\iv \iv}
     \end{array}\right ) ,
\end{equation}
is not optimal.
\end{remark}

\section{Numerical results}\label{sec:expes}
In this section we present some numerical experiments to validate the performance
of the proposed preconditioners.\\

We set $\Omega=(0,1)^{2}$, and consider a sequence of subdomain partitions
made of $N=4^{\ell}$ squares, $\ell=1,2,\ldots$, cf. Figure~\ref{fig:subdomains_ini_stru_grids} for $\ell=1,2,3,4$.
For a given subdomain partition, $\ell=1,2,\ldots$, we have tested our preconditioners on a sequence of nested structured and unstructured triangular grids made of $n=2*4^{r}$, $r=\ell,\ell+1,\ldots$.
Notice that the corresponding coarse and fine mesh sizes given by $H \approx 2^{-\ell}$, $\ell=1,2,\ldots$, and  $h \approx 2^{-(r+1/2)}$, $r=\ell,\ell+1,\ldots$, respectively.
In  Figure~\ref{fig:subdomains_ini_stru_grids}
we have reported the initial structured grids, on subdomains partitions made by $N=4^{s}$ squares, $s=1,2,3,4$, are reported. Figure~\ref{fig:subdomains_ini_unstru_grids} shows the first four refinement levels of unstructured grids on a subdomain partition made of $N=4$ squares.\\
%%%%%%%%%%%%%%%%%%%%%%%%%%%%%%%%%%%%%%%%%%
\begin{figure}[htbp]
\begin{center}
\subfigure[Initial structured grids.\label{fig:subdomains_ini_stru_grids}]{
\includegraphics[width=0.25\textwidth]{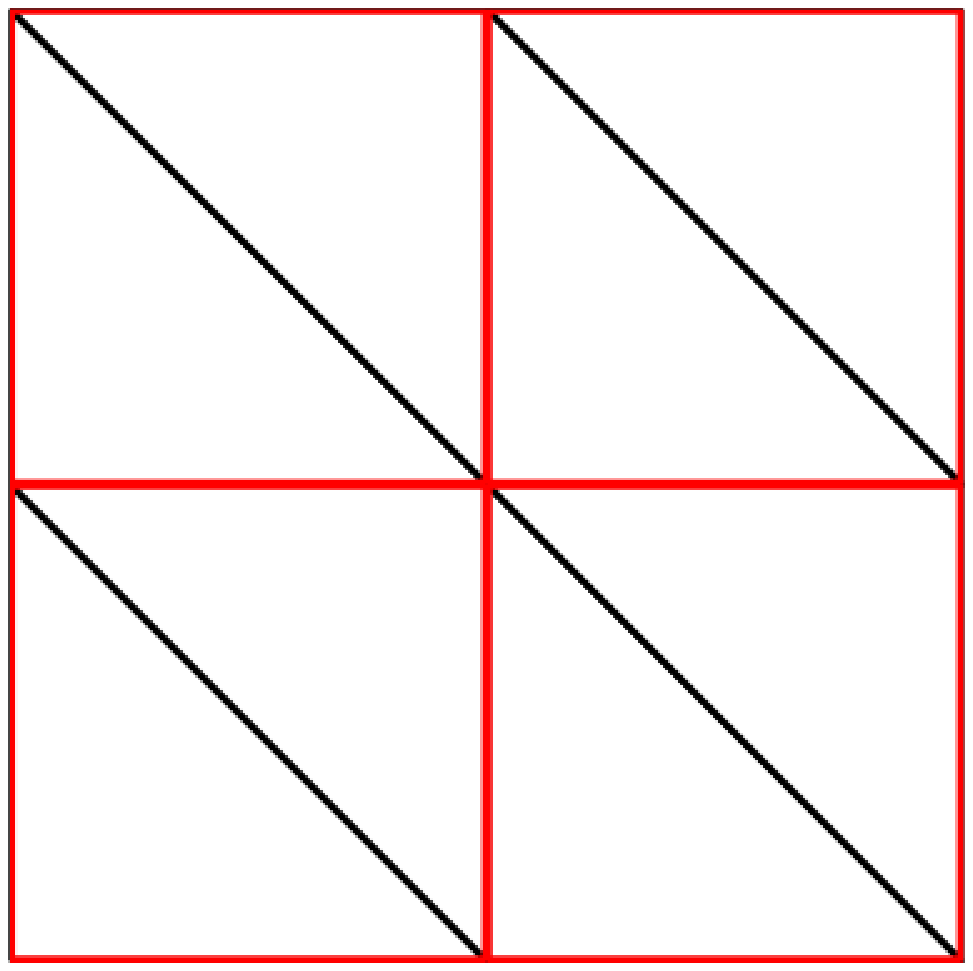}
\includegraphics[width=0.25\textwidth]{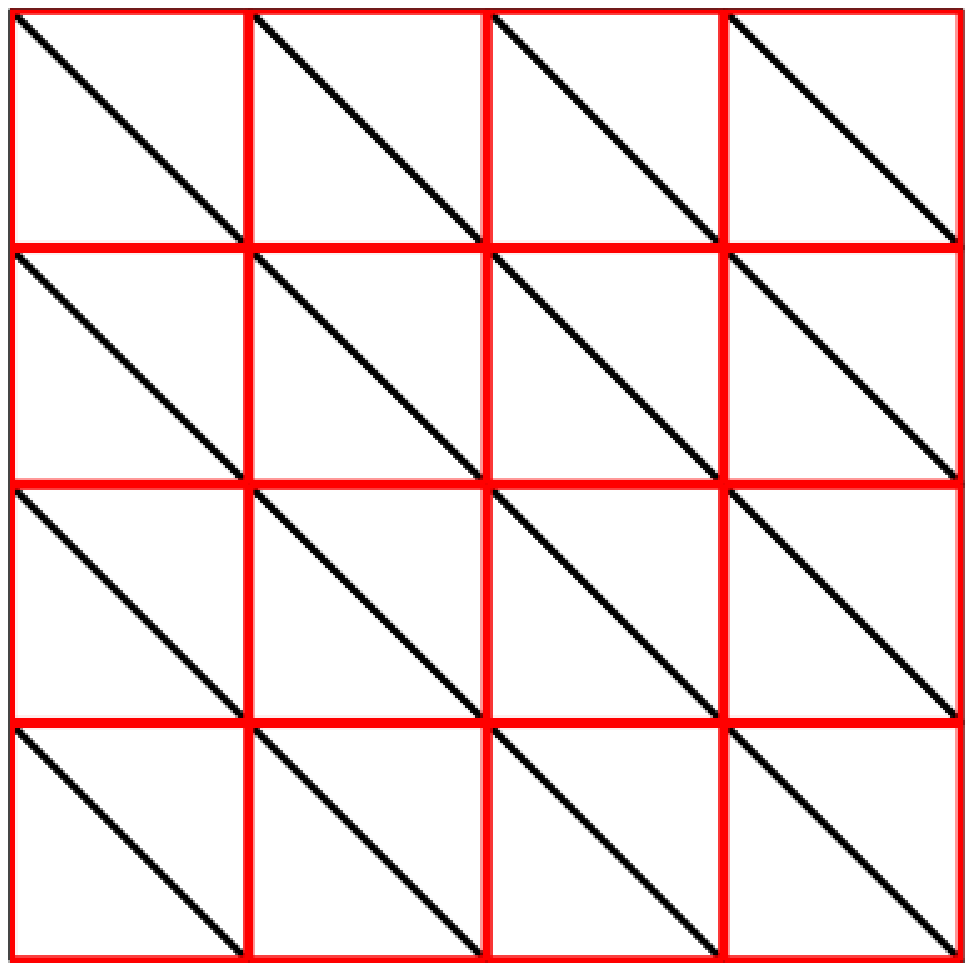}
\includegraphics[width=0.25\textwidth]{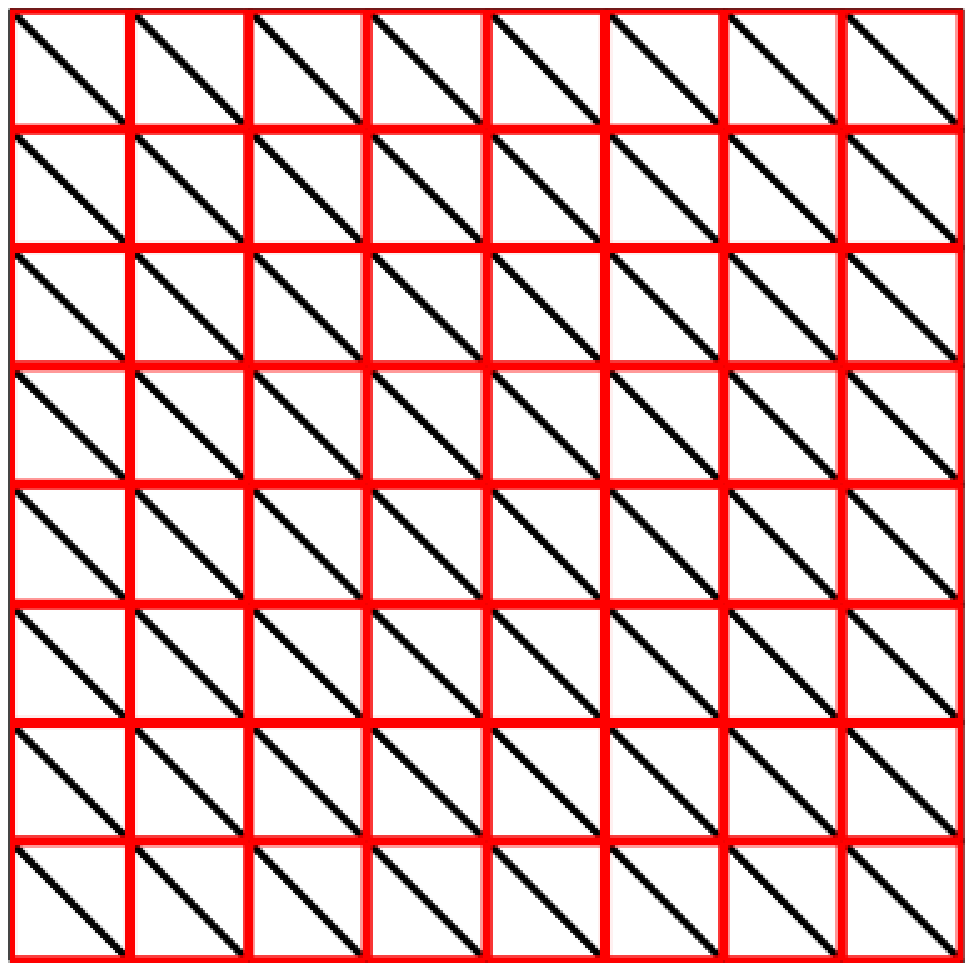}
\includegraphics[width=0.25\textwidth]{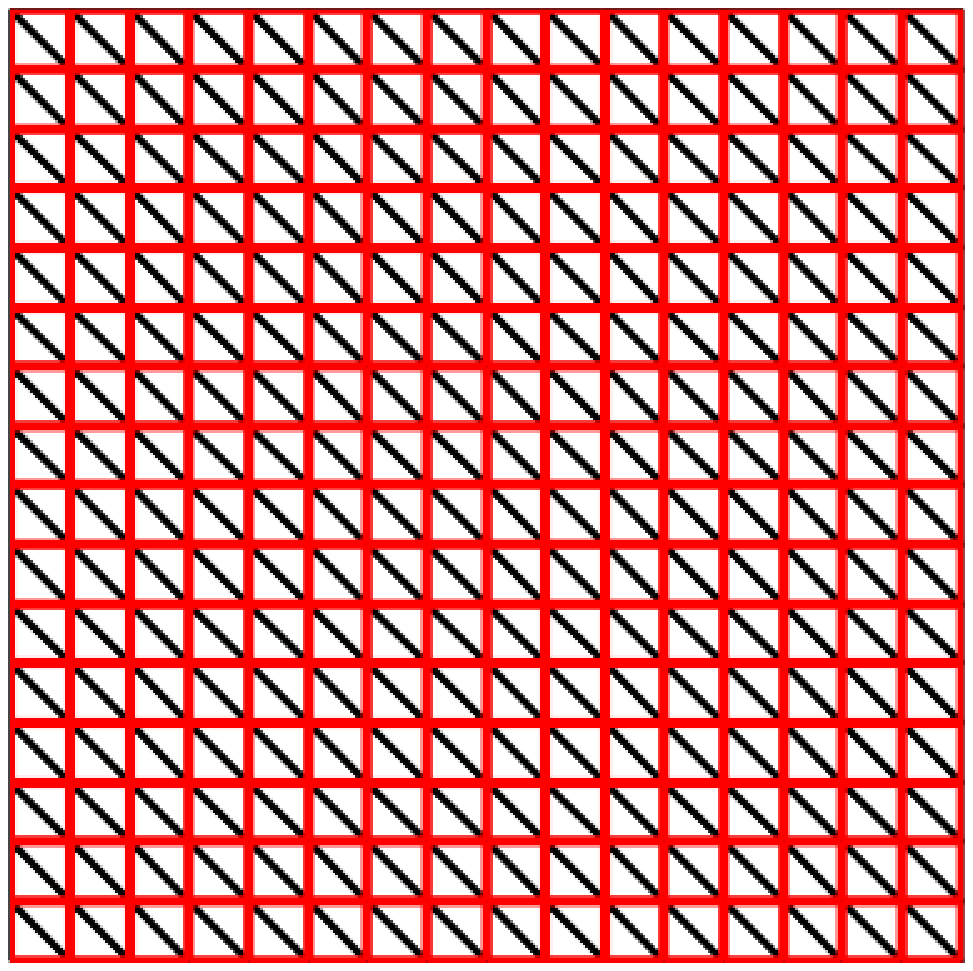}
}
\subfigure[First three refinement levels of unstructured triangular grids on a subdomain partition made of $N=4$ squares.\label{fig:subdomains_ini_unstru_grids}]{
\includegraphics[width=0.25\textwidth]{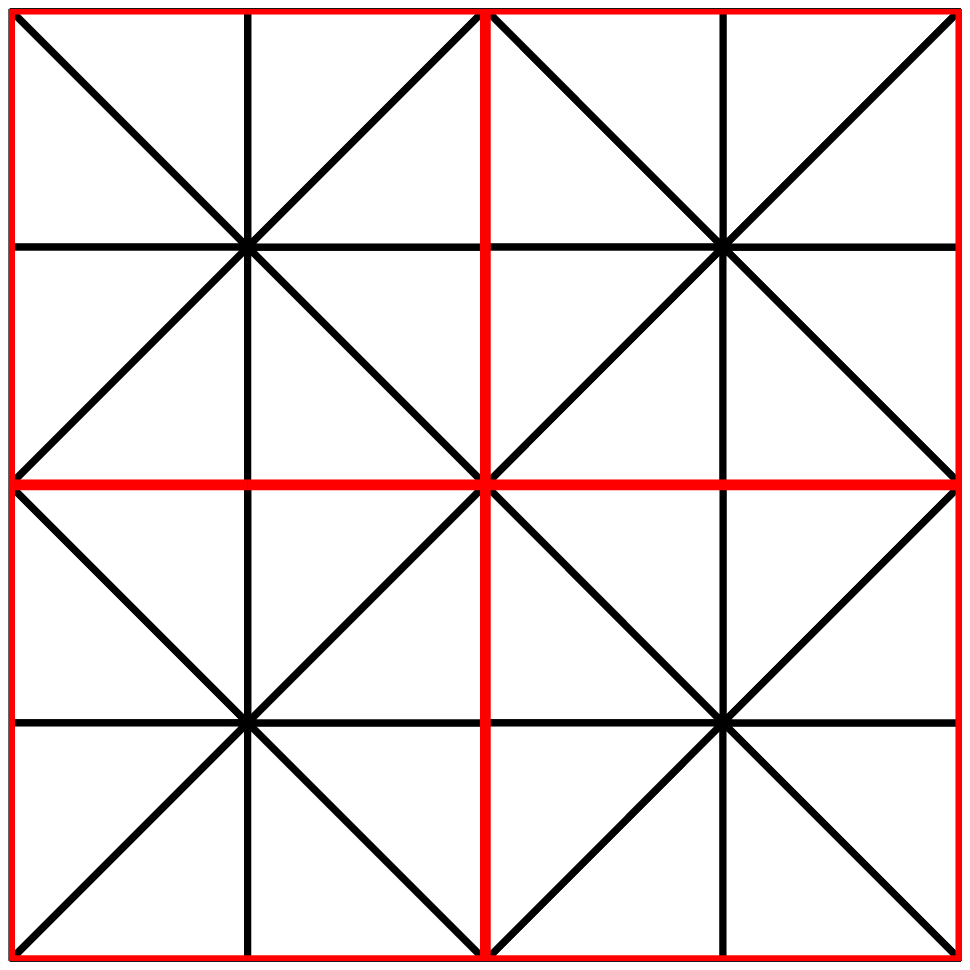}
\includegraphics[width=0.25\textwidth]{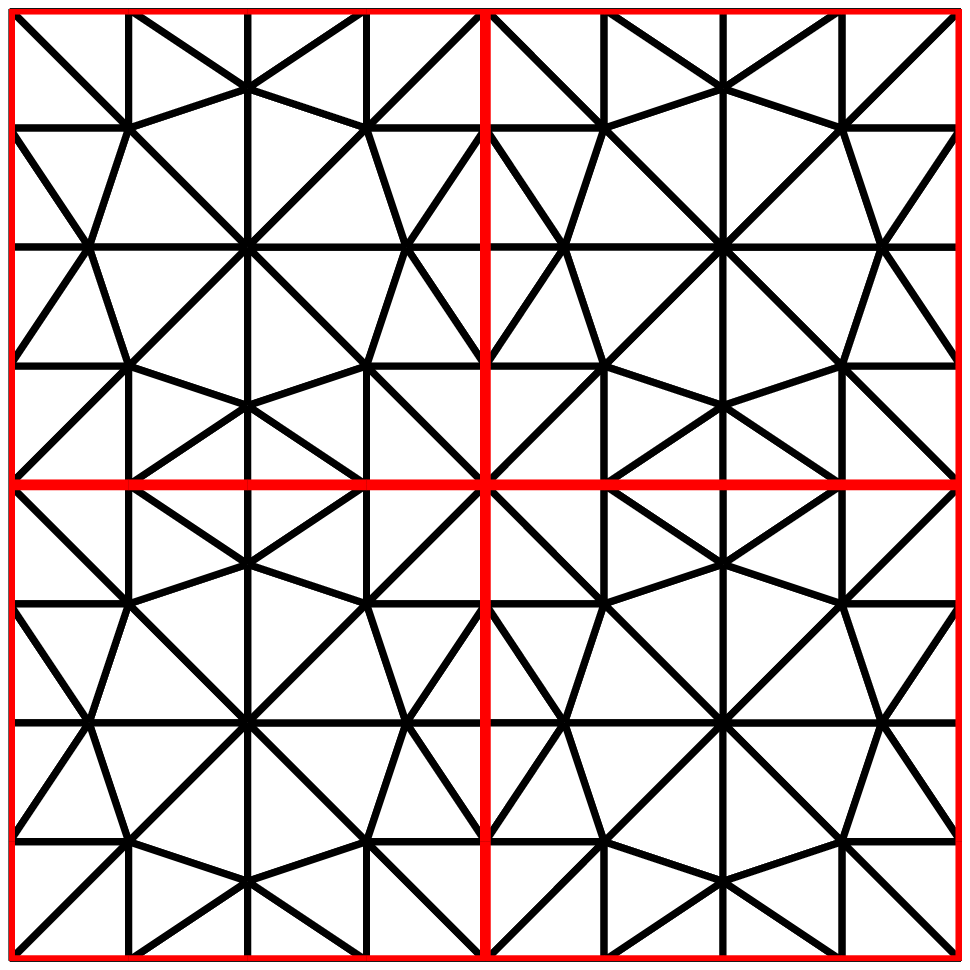}
\includegraphics[width=0.25\textwidth]{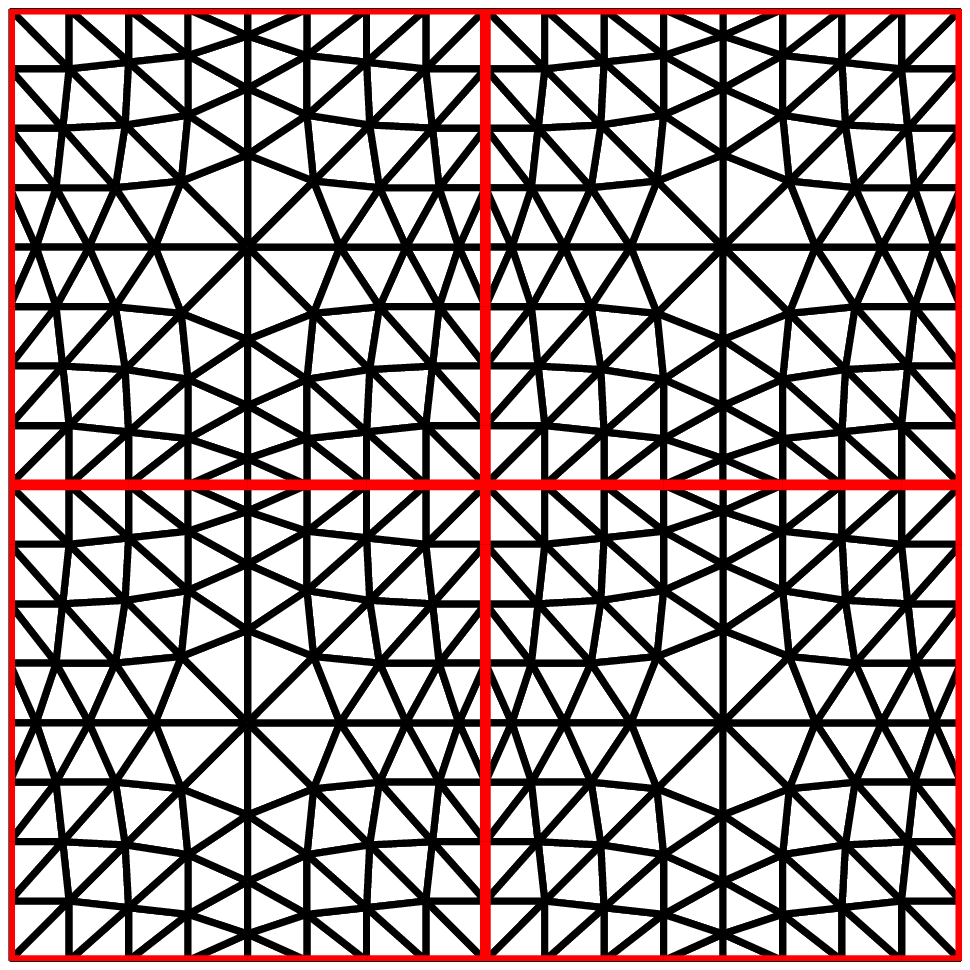}
\includegraphics[width=0.25\textwidth]{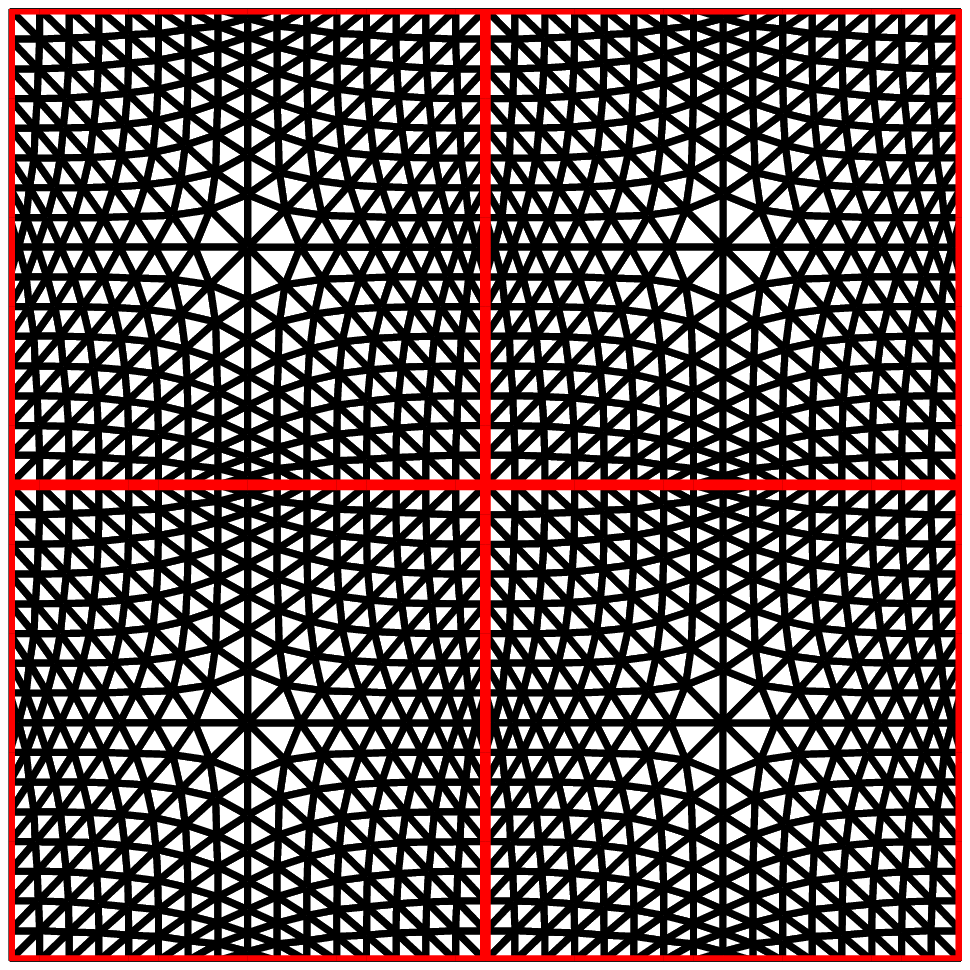}
}
\caption{Top: initial structured grids on subdomains partitions made by $N=4^{\ell}$ squares, $\ell=1,2,3,4$. Bottom: first four refinement levels of unstructured grids on a subdomain partition made of $N=4$ squares.}
\label{fig:subdomains_ini_grids}
\end{center}
\end{figure}
%%%%%%%%%%%%%%%%%%%%%%%%%%%%%%%%%%%%%%%%%%%

Throughout the section, we have solved the (preconditioned) linear system of equations by the Preconditioned Conjugate Gradient~(PCG) method with a relative tolerance set equal to $10^{-9}$.
The condition number  of the (preconditioned) Schur complement matrix has been estimated within the PCG iteration by exploiting the analogies between the Lanczos technique and the PCG method (see \cite[Sects. 9.3, 10.2]{GolubVanLoan_1996}, for more details). %\\
%{\rosa
Finally, we choose the source term in problem \eqref{prob} as
  $f(x,y)=1$, and set the penalty parameter $\alpha$ equal to $10$.
%}
%{\blu \footnote
%{\blu La scelta  f=1 era per stata fatta per capire i tests.
%Possiamo tornare al vecchio source term di paola.
%I tests riportati di seguito sono relativi a $f=1$, mesh strutturate,  $p=1$ e tolleranza
%$10^{-9}$. Non ho le matrici per mesh unstructured e $p>1$. I
%relativi tests andranno rilanciati.}}
%\\

We first present some computations that show the behavior of the condition number of the  Schur complement matrix $\bS$, cf.  \nref{eq:defS}.
%%%%%%%%%%%%%%%%%%%%%%%%%%%%%%%%%%%%%%%%%%
\begin{figure}[!htbp]
\begin{center}
\includegraphics[width=0.48\textwidth]{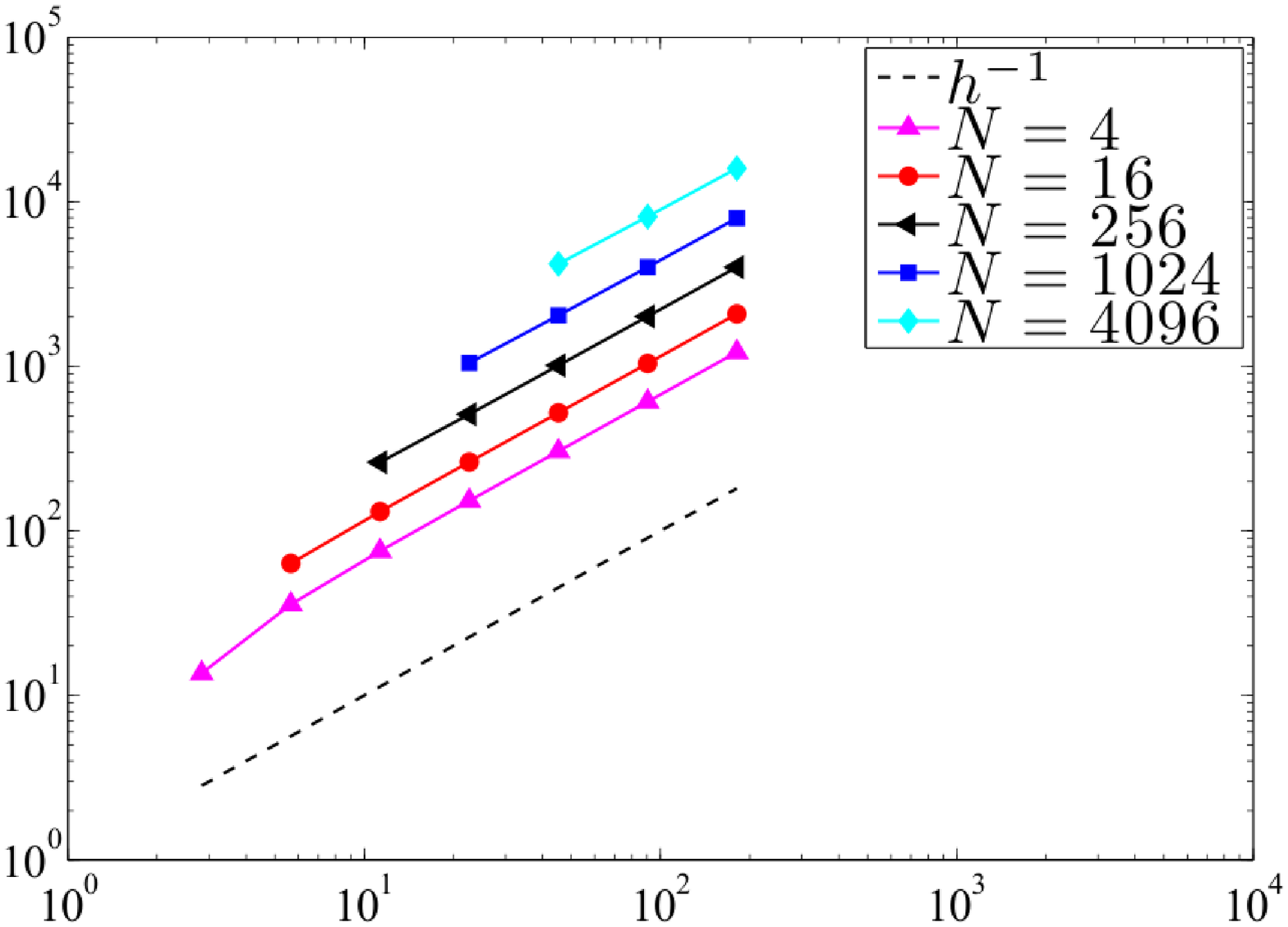}
\includegraphics[width=0.48\textwidth]{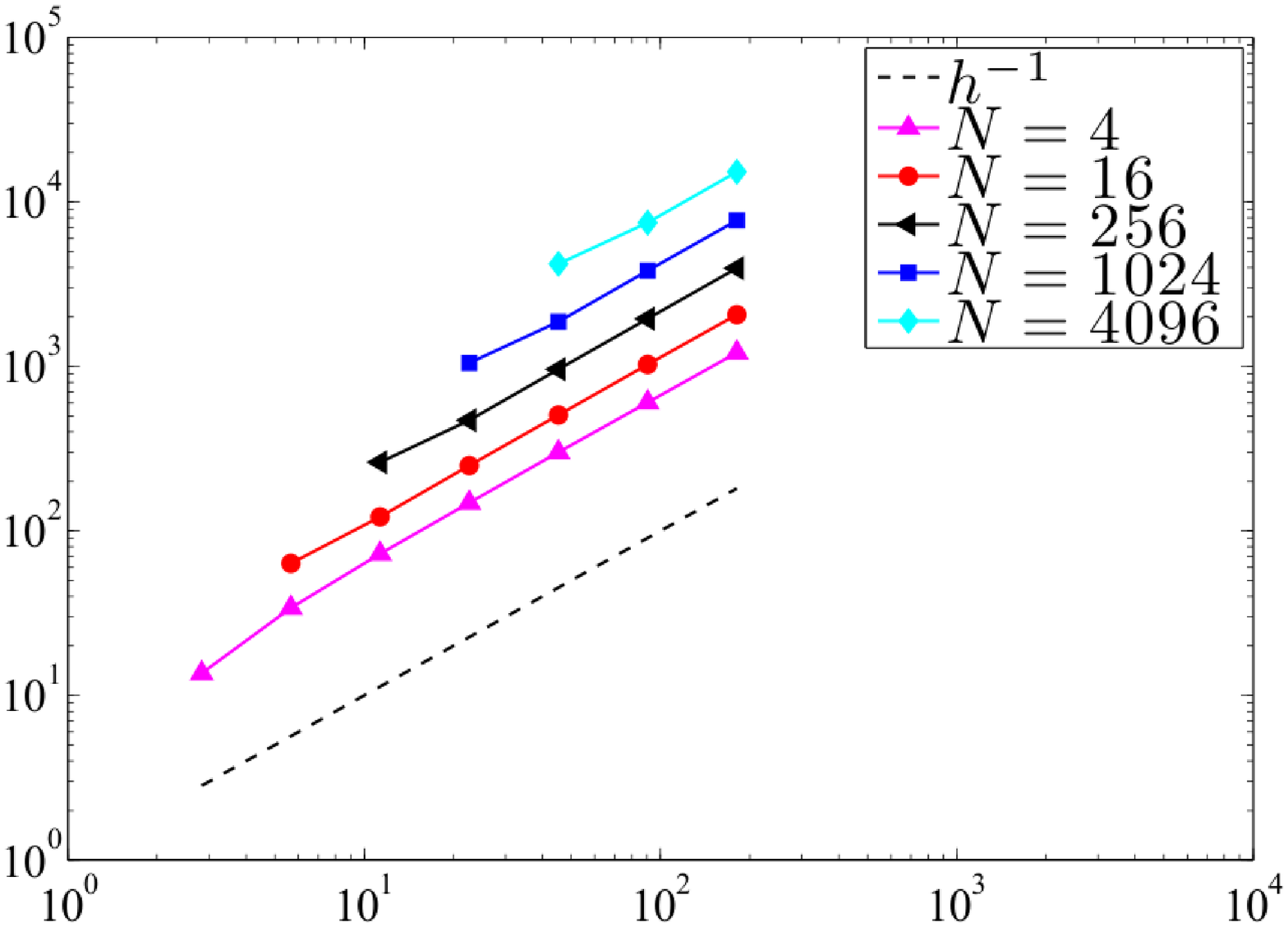}
\caption{Condition number estimate of the Schur complement matrix $\bS$ versus $1/h$ on different subdomains partitions made by $N=4^{\ell}$ squares, $\ell=1,2,3,4,5$. Structured (left) and unstructured (right) triangular grids. piecewise linear elements ($p=1$). }
\label{fig:CondS}
\end{center}
\end{figure}
%%%%%%%%%%%%%%%%%%%%%%%%%%%%%%%%%%%%%%%%%%
In Figure~\ref{fig:CondS} (log-log scale) we report,
 for different subdomains partitions made by $N=4^{\ell}$ squares, $\ell=1,2,3,4,5$,  the condition number estimate of the Schur complement matrix $\bS$, $\kappa(\bS)$, as a function of the mesh-size $1/h$.
We clearly observe that $\kappa(\bS)$ increases linearly as the mesh size $h$ goes to zero.\\

Next, we consider the preconditioned linear system of equations
\begin{equation*}
\bP^{-1}\widetilde{\bS} \, \widetilde{\bu}= \bP^{-1} \widetilde{\bg},
\end{equation*}
and test the performance of the preconditioners $\bP$ and $\bP_{\star}$
(cf. \nref{P2} and \nref{eq:defP1}, respectively). Throughout the section, the  action of the preconditioner has been computed with a direct solver. \\

In the first set of experiments, we consider piecewise linear elements
($p=1$), and compute the condition number estimates when varying the number of subdomains and the  mesh size. Table~\ref{tab:precP1P2_hversion_TS} shows the  condition number estimates  increasing the number of subdomains $N$ and the number of elements $n$ of the fine mesh.
In  Table~\ref{tab:precP1P2_hversion_TS}  we also report (between parenthesis) the ratio between the  condition number of the preconditioned system and $\left(1+\log(H/h)\right)^2$ (between parenthesis).
These results have been obtained on a sequence of structured triangular grids as the ones shown in Figure~\ref{fig:subdomains_ini_stru_grids}.  Results reported in Table~\ref{tab:precP1P2_hversion_TS}~(top) refers to the performance of the preconditioner  $\bP$, whereas the analogous results obtained with the preconditioner  $\bP_{\star}$ are shown in Table~\ref{tab:precP1P2_hversion_TS}~(bottom).
%==================================================
\begin{table}[!htbp]
\begin{center}
\begin{tabular}{lrrrrr}
&\multicolumn{5}{c}{Preconditioner $\bP$}\\
\hline$N\downarrow \; n\rightarrow\;$
&\multicolumn{1}{c}{$n=128$}
&\multicolumn{1}{c}{$n=512$}
&\multicolumn{1}{c}{$n=2048$}
&\multicolumn{1}{c}{$n=8192$}
&\multicolumn{1}{c}{$n=32768$}\\
\hline\\[-0.3cm]
$N=   16$   &	3.11 (0.74) &	4.88 (0.65) &	7.50 (0.64) &	10.84 (0.64) &	14.79 (0.64) \\
$N=   64$  	& - 	    &	3.30 (0.79) &	5.25 (0.70) &	 8.00 (0.68) &	11.42 (0.67) \\
$N=  256$  	& -         & -             &   3.35 (0.81) &	 5.36 (0.72) &	 8.16 (0.70) \\
$N= 1024$  	& - 	    & -             & - 	    &     3.37 (0.81) &	 5.39 (0.72) \\
 \hline	\\
&\multicolumn{5}{c}{Preconditioner $\bP_{\star}$}\\
\hline
&\multicolumn{1}{c}{$n=128$}
&\multicolumn{1}{c}{$n=512$}
&\multicolumn{1}{c}{$n=2048$}
&\multicolumn{1}{c}{$n=8192$}
&\multicolumn{1}{c}{$n=32768$}\\
\hline\\[-0.3cm]
$N=   16$  &	2.26 (0.54) &	4.04 (0.54) &	7.01 (0.60) &	11.00 (0.65) &	15.83 (0.68) 	\\
$N=   64$  	& - 	 &	2.42 (0.58) &	4.49 (0.60) &	7.85 (0.67) &	12.28 (0.72) 	\\
$N=  256$  	& - 	 	& - 	 &	2.47 (0.59) &	4.60 (0.62) &	8.07 (0.69) 	\\
$N= 1024$  	& - 	 	& - 	 	& - 	 	&2.48 (0.60) &	4.63 (0.62) 	\\
\hline
\end{tabular}
\caption{Preconditioner $\bP$ (top) and $\bP_{\star}$ (bottom). Condition number estimates and ratio between the  condition number of the preconditioned system and $\left(1+\log(H/h)\right)^2$ (between parenthesis). Structured triangular grids, piecewise linear elements ($p=1$). }%
\label{tab:precP1P2_hversion_TS}%
\end{center}%
\end{table}%
%==================================================
We have repeated the same set of experiments on the sequence of unstructured triangular grids (cf. Figure~\ref{fig:subdomains_ini_unstru_grids}). The computed results are shown  in Figure~\ref{tab:precP1P2_hversion_TU}.
As before, between parenthesis we report
ratio between the  condition number of the preconditioned system and
$\left(1+\log(H/h)\right)^2$. As expected, a logarithmic growth is
clearly observed for both preconditioner $\bP$ and $\bP_{\star}$. \\
%
%, even if for the preconditioner $\bP$ such a behavior seems to be achieved only asimptotoically.\\

%==================================================
\begin{table}[!htbp]
\begin{center}
\begin{tabular}{lrrrrr}
&\multicolumn{5}{c}{Preconditioner $\bP$}\\
\hline
&\multicolumn{1}{c}{$n=128$}
&\multicolumn{1}{c}{$n=512$}
&\multicolumn{1}{c}{$n=2048$}
&\multicolumn{1}{c}{$n=8192$}
&\multicolumn{1}{c}{$n=32768$}\\
\hline\\[-0.3cm]
$N=   16$  &	2.87 (0.69) &	4.69 (0.63) &	7.35 (0.63) &	10.68 (0.63) &	14.62 (0.63) 	\\ 
$N=   64$  & - 	 	&3.05 (0.73) &	5.01 (0.67) &	7.75 (0.66) &	11.13 (0.66) 	\\ 
$N=  256$ & - 	 	& - 	 	&3.09 (0.74) &	5.08 (0.68) &	7.89 (0.67) 	\\ 
$N= 1024$ & - 	 	& - 	 	& - 	 	&3.11 (0.75) &	5.11 (0.68) 	\\ 
 \hline	\\
&\multicolumn{5}{c}{Preconditioner $\bP_{\star}$}\\
\hline
&\multicolumn{1}{c}{$n=128$}
&\multicolumn{1}{c}{$n=512$}
&\multicolumn{1}{c}{$n=2048$}
&\multicolumn{1}{c}{$n=8192$}
&\multicolumn{1}{c}{$n=32768$}\\
\hline\\[-0.3cm]
$N=   16$  &	1.84 (0.44) &	3.24 (0.43) &	5.51 (0.47) &	8.44 (0.50) &	12.00 (0.52) 	\\
$N=   64$  &	 - 	 	& 2.01 (0.48) &	3.77 (0.50) &	6.35 (0.54) &	9.76 (0.58) 	\\
$N=  256$  	& - 	 	& - 	 	& 2.04 (0.49) &	3.90 (0.52) &	6.58 (0.56) 	\\
$N= 1024$  	& - 	 	& - 	 	& - 	 	&2.05 (0.49) &	3.93 (0.53)	\\
\hline
\end{tabular}
\caption{Preconditioner $\bP$ (top) and $\bP_{\star}$ (bottom). Condition number estimates and ratio between the  condition number of the preconditioned system and $\left(1+\log(H/h)\right)^2$ (between parenthesis). Unstructured triangular grids, piecewise linear elements ($p=1$). }%
\label{tab:precP1P2_hversion_TU}%
\end{center}%
\end{table}%
%==================================================

Next, always with $p=1$, we present some computations that show that the preconditioner $\bP_{D}$ defined as in \nref{eq:defP1diag}, i.e.,  the block-diagonal version of the preconditioner $\bP_{\star}$, is not optimal (cf. Remark \ref{rem:defP1diag}). More precisely, in Table~\ref{tab:precPdaisy_hversion_TS_TU} we report the condition number estimate of the preconditioned system when decreasing $H$ as well as $h$. Table~\ref{tab:precPdaisy_hversion_TS_TU} also shows (between parenthesis) the ratio between $\kappa(\bP_{D} \widetilde{\bS})$ and $Hh^{-1}$.
We can clearly observe that on both structured and unstructured mesh configurations, the ratio between $\kappa(\bP_{D} \widetilde{\bS})$ and $Hh^{-1}$
remains substantially constant as $H$ and $h$ vary, indicating that
the preconditioner $\bP_{D}$ is not optimal. \\

%==================================================
\begin{table}[!htbp]
\begin{center}
\begin{tabular}{lrrrrr}
&\multicolumn{5}{c}{Structured triangular grids}\\
\hline
&\multicolumn{1}{c}{$n=128$}
&\multicolumn{1}{c}{$n=512$}
&\multicolumn{1}{c}{$n=2048$}
&\multicolumn{1}{c}{$n=8192$}
&\multicolumn{1}{c}{$n=32768$}\\
\hline\\[-0.3cm]
$N=   16$ &	11.51 (4.07) & 23.19 (4.10) & 47.40 (4.19) & 95.21 (4.21) & 190.69 (4.21) \\
$N=   64$ & - 	             & 11.58 (4.09) & 23.03 (4.07) & 47.16 (4.17) & 95.02 (4.20) \\
$N=  256$ & - 	             & - 	    & 11.55 (4.08) & 22.96 (4.06) & 47.12 (4.16) \\
$N= 1024$ & - 	             & - 	    & - 	   & 11.44 (4.04) & 22.88 (4.04) \\
\hline\\
&\multicolumn{5}{c}{Unstructured triangular grids }\\
\hline
&\multicolumn{1}{c}{$n=128$}
&\multicolumn{1}{c}{$n=512$}
&\multicolumn{1}{c}{$n=2048$}
&\multicolumn{1}{c}{$n=8192$}
&\multicolumn{1}{c}{$n=32768$}\\
\hline\\[-0.3cm]
$N=   16$  &	9.45 (3.34) &	18.63 (3.29) &	39.13 (3.46) &	75.38 (3.33) &	148.93 (3.29) 	\\
$N=   64$  	& - 	 	& 8.93 (3.16) &	18.30 (3.24) &	38.88 (3.44) &	78.82 (3.48) 	\\
$N=  256$  	& - 	 	& - 	 	& 8.80 (3.11) &	17.85 (3.15) &	38.59 (3.41) 	\\
$N= 1024$  	& - 	 	& - 	 	& - 	 &	8.75 (3.10) &	17.64 (3.12) 	\\
\hline\\
\end{tabular}
\caption{Preconditioner $\bP_{D}$. Condition number estimates and ratio between $\kappa(\bP_{D} \widetilde{\bS})$ and $Hh^{-1}$ (between parenthesis). Structured (top) and unstructured (bottom) triangular grids, piecewise linear elements ($p=1$).}%
\label{tab:precPdaisy_hversion_TS_TU}%
\end{center}%
\end{table}%
%==================================================

%==================================================

Finally, we present some computations obtained with high-order
elements. As before, we consider a subdomain partition made of
$N=4^{\ell}$ squares, $\ell=1,2,\ldots$,
(cf. Figure~\ref{fig:subdomains_ini_stru_grids} for $\ell=1,2,3$).
 In this set of experiments, the subdomain partition coincides with
 the fine grid, i.e., $H=h$ , and on each element we consider the
 space of polynomials of degree $p=2,3,4,5,6$ in each  coordinate
 direction.

Table~\ref{tab:nonprec_pversion_CART} shows the condition number estimate of the non-preconditioned Schur complement matrix and the CG iteration counts.%=================================================
\begin{table}[!htbp]
\begin{center}
\begin{tabular}{llllll}
\hline
$N=n$
&\multicolumn{1}{c}{$p=2$}
&\multicolumn{1}{c}{$p=3$}
&\multicolumn{1}{c}{$p=4$}
&\multicolumn{1}{c}{$p=5$}
&\multicolumn{1}{c}{$p=6$}\\
\hline\\[-0.3cm]
$4$   & 5.1e+1 ( 5) 	& 2.7e+2 ( 8) 	& 6.2e+2 (13) 	& 1.4e+3 (18) 	& 3.4e+3 (28) 	\\ 
$16$   & 3.2e+2 (22) 	& 8.4e+2 (42) 	& 2.0e+3 (69) 	& 4.6e+3 (101) 	& 1.1e+4 (153) 	\\ 
$64$   & 1.2e+3 (90) 	& 3.2e+3 (150) 	& 7.6e+3 (231) 	& 1.8e+4 (312) 	& 4.3e+4 (446) 	\\ 
$256$   &  4.7e+3 (195) 	& 1.3e+4 (294) 	& 3.0e+4 (462) 	& 7.0e+4 (634) 	& 1.7e+5 (886) 	\\ 
\hline
\end{tabular}
\caption{Condition
  number estimates $\kappa(\bS)$ and CG iteration counts (between
  parenthesis).  Cartesian grids.
}%
\label{tab:nonprec_pversion_CART}%
\end{center}%
\end{table}%
%=================================================
We have run the same  set of experiments employing the  preconditioners $\bP$ and $\bP_{\star}$, and the results are reported  in Table~\ref{tab:precP1P3_pversion_CART} .
We clearly observe that, as predicted, for a fixed mesh configuration the condition number of the preconditioned system grows logarithmically as the polynomial approximation degree increases.
Comparing these results with the analogous ones reported in  Table~\ref{tab:nonprec_pversion_CART}, it can be inferred that both the preconditioners $\bP$ and $\bP_{\star}$ are efficient in reducing the condition number of the Schur complement matrix. 
 %=================================================
\begin{table}[!htbp]
\begin{center}
\begin{tabular}{lrrrrr}
&\multicolumn{5}{c}{Preconditioner $\bP$}\\
\hline
$N=n$
&\multicolumn{1}{c}{$p=2$}
&\multicolumn{1}{c}{$p=3$}
&\multicolumn{1}{c}{$p=4$}
&\multicolumn{1}{c}{$p=5$}
&\multicolumn{1}{c}{$p=6$}\\
\hline\\[-0.3cm]
$N=    4$ &	7.14 (1.25) &	9.04 (0.88) &	12.06 (0.85) &	14.15 (0.79) &	16.48 (0.78) 	\\ 
$N=   16$ &	9.24 (1.62) &	9.93 (0.97) &	15.25 (1.07) &	15.99 (0.90) &	20.25 (0.96) 	\\ 
$N=   64$ &	10.03 (1.76) &	10.14 (0.99) &	16.34 (1.15) &	16.57 (0.93) &	21.53 (1.02) 	\\ 
$N=  256$ &	10.24 (1.80) &	10.19 (1.00) &	16.61 (1.17) &	16.71 (0.94) &	21.84 (1.04) 	\\ 
\hline\\
&\multicolumn{5}{c}{Preconditioner $\bP_{\star}$}\\
\hline
$N=n$
&\multicolumn{1}{c}{$p=2$}
&\multicolumn{1}{c}{$p=3$}
&\multicolumn{1}{c}{$p=4$}
&\multicolumn{1}{c}{$p=5$}
&\multicolumn{1}{c}{$p=6$}\\
\hline\\[-0.3cm]
$N=    4$ &	1.88 (0.33) &	2.56 (0.25) &	3.75 (0.26) &	4.64 (0.26) &	5.70 (0.27) 	\\ 
$N=   16$ &	4.60 (0.81) &	5.23 (0.51) &	8.71 (0.61) &	9.38 (0.53) &	12.25 (0.58) 	\\ 
$N=   64$ &	6.18 (1.09) &	6.03 (0.59) &	10.35 (0.73) &	10.79 (0.61) &	14.33 (0.68) 	\\ 
$N=  256$ &	6.55 (1.15) &	6.25 (0.61) &	10.83 (0.76) &	11.20 (0.63) &	14.94 (0.71) 	\\ 
\hline
\\
\end{tabular}
\caption{Preconditioner $\bP$ (top), $\bP_{\star}$ (bottom).  Condition number estimates and ratio between the  condition number of 
the preconditioned system and $\left(1+\log(p^{2})\right)^2$ (between
parenthesis). Cartesian grids.} %
\label{tab:precP1P3_pversion_CART}%
\end{center}%
\end{table}%
%================================================

\appendix
\section{Appendix}\label{app0}
In this section, we report the proofs of Lemma~\ref{lembsp2} and Lemma~\ref{lembsp}.

\

 In the following, for $E\subset \Omega_\ell$ subdomain edge we will
 make explicit use of the space $H^{s}_0(E)$, $0<s<1/2$, which is defined as the subspace of
 those functions $\eta$ of $H^s(E)$ such that
 the function $\bar \eta \in L^2(\partial\omk)$ defined as $\bar \eta = \eta$ on $E$ and $\bar \eta = 0$ on $\partial\omk\setminus E$
 belongs to $H^s(\partial\omk)$.
 The space  $H^{s}_0(E)$ is endowed with the norms
 \[
 \| \eta \|_{H^s_0(E)} = \| \bar \eta \|_{H^s(\partial\omk)}.\]
 We recall that for $s < 1/2$ the spaces $H^s(E)$ and
 $H^{s}_0(E)$ coincide as sets and have equivalent norms. However, the
 constant in the norm equivalence goes to infinity as $s$ tends to $1/2$.
 In particular on the reference segment $\widehat E =(0,1)$, for all $\varphi \in
H^{s}(\widehat{E})$ and for all $\beta \in \RR$, the following bound can be shown (see \cite{Bsubstr})
\begin{equation*}
| \varphi |_{H^{s}_0(\widehat{E})}
\lesssim  \frac 1 {1/2-s} \| \varphi - \beta \|_{H^{s}(\widehat{E})}
+ \frac 1 {\sqrt{1/2-s}}
|\beta|,
\end{equation*}
 which, provided $\varphi \in H^{1/2}(\widehat E)$, implies the bound
 \begin{equation}
| \varphi |_{H^{s}_0(\widehat{E})}\lesssim
\frac 1 {1/2-s} \| \varphi - \beta \|_{H^{1/2}(\widehat{E})}
+ \frac 1 {\sqrt{1/2-s}}
|\beta|.\label{penultima}
\end{equation}

Prior to give the proofs of Lemmas \ref{lembsp2} and \ref{lembsp}, we start by observing that the following result,  that corresponds to the $hp$-version of  \cite[Lemma~3.1]{Bsubstr}, holds.
\begin{lemma} \label{lemmainf}
%(\cite[Lemma~3.1]{Bsubstr})
%Let $\Lineari \in \tracehl$ and l
Let $E=(a,b)$ be a subdomain edge of  $\omk$.
Then, for all $\eta \in\Phi_{\ell}(E)$, the
following bounds hold:
\begin{itemize}
\item[{\em (i)}]
\begin{equation}
\label{eq:19}
(\eta(a) - \eta(b))^2 \lesssim \left(1+\log{\left(\frac{\Ho\,\po^{2}}{\ho}\right)}\right) | \eta |^2_{H^{1/2}(E)}.
\end{equation}
\item[{\em (ii)}]  If  $\eta(x)=0$ at some
$x \in E$ it holds
\begin{equation}\label{eq:18}
\| \eta \|^2_{L^\infty(E)} \lesssim \left(1+\log{\left(\frac{\Ho\,\po^{2}}{\ho}\right)}\right) | \eta |^2_{H^{1/2}(E)}.
\end{equation}
\item[{\em (iii)}] if $\eta \in\Phi^{0}_{\ell}(E)$, we have
\begin{equation}
\label{eq:4}
\| \eta \|_{H^{1/2}_{00}(E)}^2\lesssim \left(1+\log{\left(\frac{\Ho\,\po^{2}}{\ho}\right)}\right)  |
\eta |_{H^{1/2}(\partial \omk)}^2.
\end{equation}
\end{itemize}
\end{lemma}
%
%Lemma~\ref{lemmainf} is a generalization to the $hp$-version of \cite[Lemma~3.1]{Bsubstr}

\begin{proof}
We first show {\em (ii)}. Notice that since $E$ is an arbitrary subdomain edge, $E\subset \partial\omk$, and so $|E|\simeq H_{\ell}$.  We claim that for any $\varphi \in H^{1/ 2+\varepsilon}(E)$ the following inequality holds:
 \begin{equation}
   \label{eq:2}
   \| \varphi  \|^2_{L^{\infty}(E)} \lesssim H^{-1}_{\ell} \| \varphi  \|_{L^2(E)}^2 + \frac {H^{2\varepsilon}_{\ell}} {\varepsilon} | \varphi |^2_{H^{1/ 2+\varepsilon}(E)}.
 \end{equation}
 To show \eqref{eq:2} one needs to trace the constants in the Sobolev imbedding between $H^{1/2+\varepsilon}(E)$ and $L^{\infty}(E)$.
 Let $\widehat{E} =]0,1[$ be the reference unit segment. Then, for any $\hat{\varphi} \in H^{1/ 2+\varepsilon}(\widehat{E})$, the continuity constant of the injection $H^{1/2+\varepsilon}(\widehat{E}) \subset
L^{\infty}(\widehat{E})$  depends on $\varepsilon$ as follows (see \cite[Appendix]{Bsubstr}, for details)
\[
   \|\hat{ \varphi} \|^2_{L^{\infty}(\widehat{E})} \lesssim  \|\hat{ \varphi} \|_{L^2(\widehat{E})}^2 +
  \frac 1 {\varepsilon} | \hat{\varphi} |^2_{H^{1/ 2+\varepsilon}(\widehat{E})}.
\]
A scaling argument using $|E|\simeq H_{\ell}$ leads to \eqref{eq:2}.

Let now $\eta \in \Phi_h$ and $\beta \in \mathbb{R}$ an arbitrary
constant.
Using the inverse inequality \eqref{inv:E}, we have
 \begin{equation}
   \label{eq:3}
   \begin{aligned}
   \frac  {\Ho^{2\varepsilon}} {\varepsilon} | \eta - \beta |^2_{H^{1/ 2+\varepsilon}(E)} &=
\frac  {\Ho^{2\varepsilon}} {\varepsilon} | \eta |^2_{H^{1/ 2+\varepsilon}(E)}
\lesssim \frac {\Ho^{2\varepsilon}\, \po^{4\varepsilon}\ho^{-2\varepsilon}}{\varepsilon} | \eta |^2_{H^{1/ 2}(E)} &&\\
& \lesssim \log{\left(\frac{\Ho \,\po^{2}}{\ho}\right)} | \eta |^2_{H^{1/ 2}(E)}\;, &&
\end{aligned}
 \end{equation}
where in the last step we have taken $\varepsilon = 1/\log(H_{\ell}\po^{2}/\ho)$ and used the fact that $s^{1/\log(s)}=e$. Applying now
 inequality
\eqref{eq:2} to $\varphi=\eta - \beta$ together with the above estimate \eqref{eq:3} yields:
\begin{equation}
  \label{eq:27}
  \| \eta - \beta \|_{L^\infty(E)}^2 \lesssim H_{\ell}^{-1} \| \eta- \beta \|^2_{L^2(E)} +
  \log \left( \frac {H_{\ell}\po^2} {\ho } \right)| \eta |_{H^{1/2}(E)}^2.
\end{equation}
Following \cite{BPS}, let  $\beta$ be the average over $E$ of $\eta$ (or the $L^{2}$- projection onto the space $\mathbb{P}^{0}(E))$ of constants functions over $E$. Poincar\`e-Friederichs inequality (or standard approximation results) give
 \begin{equation}
   \label{eq:20}
H_{\ell}^{-1/2} \| \eta - \beta \|_{L^2(E)} \lesssim | \eta |_{H^{1/2}(E)}
 \end{equation}
which yields
 \begin{equation}
   \label{eq:21}
   \| \eta - \beta \|^2_{L^\infty(E)} \lesssim \left( 1 + \log \left( \frac {H_{\ell}\po^{2}} {\ho}\right) \right)| \eta |^2_{H^{1/2}(E)}.
 \end{equation}
 The proof of \eqref{eq:18} is concluded by noticing that if  $\eta(x) = 0$ for some $x \in \bar{E}$ then it follows
\begin{equation}
  \label{eq:21b}
    | \beta | \lesssim \| \eta - \beta \|_{L^\infty(E)}
\end{equation}
 which yields \eqref{eq:18} using triangular inequality.

The proof of \eqref{eq:19} follows by applying the estimate \eqref{eq:18} to the function $\eta-\eta(a)$, which by hypothesis vanishes at $a\in \bar{E}$.

 To show {\em (iii)}, we first notice that for $\etao \in
 \Phi^{o}_{\ell}(E)$, we can always construct an extension
 $\widetilde{\eta}_o$ such that
 \begin{equation*}
 \begin{aligned}
&\widetilde{\eta}_o =\etao
\; \textrm{ on  $E$}
&& \quad \widetilde{\eta}_o=0
\; \textrm{ on  $\partial\omk \setminus E$}\;.
\end{aligned}
\end{equation*}
 Using now
%(\ref{h1200toh12})
 the inverse
inequality  \eqref{inv:O},  we obtain the following bounds
\begin{equation}
\| \etao \|_{H^{1/2}_{00}(E)}  \lesssim | \widetilde{\eta}_o |_{H^{1/2}(\partial\omk) }
\lesssim \po^{2\varepsilon} \ho^{-\varepsilon} | \tilde\eta_{o} |_{H^{1/2 - \varepsilon}(\partial\omk)}
\lesssim \po^{2\varepsilon} \ho^{-\varepsilon}| \etao
|_{H^{1/2-\varepsilon}_0(E)},\label{36}
\end{equation}
where the second inequality follows from the boundedness from $H^{1/2}_{00}(E)$ to $H^{1/2}(\partial\omk)$ of the extension by $0$.

To estimate now the $H^{1/2-\varepsilon}_0(E)$ seminorm of $\etao$ we observe that
\nref{penultima} rescales as
\[
| \varphi |_{H^{1/2-\varepsilon}_0(E)} \lesssim \frac{\Ho^{\varepsilon}} \varepsilon \left(
  \Ho^{-1/2} \| \varphi - \beta \|_{L^2(E)} + | \varphi - \beta
|_{H^{1/2}(E)} \right)  + \frac{\Ho^{\varepsilon}}{\sqrt{\varepsilon}} | \beta |.
\]
Taking now $\varphi = \etao$ and choosing $\beta$ as its average  on $E$, the first term on the right hand side above is bounded by means of
Poincar\'e-Friederichs inequality, and the second by means of estimate
 (\ref{eq:21b}), which holds since $\etao(a) = 0$. Hence, we
get
\[
\|  \etao \|_{H^{1/2}_{00}(E)} \lesssim  \frac{\Ho^{\varepsilon} \po^{2\varepsilon}\ho^{-\varepsilon}} \varepsilon
| \etao |_{H^{1/2}(E)} +  \frac{\Ho^{\varepsilon}\po^{2\varepsilon} \ho^{-\varepsilon}}{\sqrt{\varepsilon}}
\| \etao - \beta \|_{L^\infty(E)}.
\]
Arguing as before and taking $\varepsilon = 1/\log(\Ho\po^{2}/\ho)$, and using bound
(\ref{eq:21})
we obtain
\[
\|  \etao \|_{H^{1/2}_{00}(E)} \lesssim  \left(1 + \log \frac {\Ho\, \po^{2}}{\ho}\right) | \etao |_{H^{1/2}(E)}.
\]
Finally, since
$\sum_{E\subset \partial\omk} | \cdot |_{H^{1/2}(E)}^2 \lesssim | \cdot
|_{H^{1/2}(\partial\omk)}^2,$
by squaring and taking the sum over $E\subset \partial\omk$,
we obtain \eqref{eq:4}.
\end{proof}
%%%%%%%%%%%%%%%%%%%%%%%%%%%%%%%%%%

\

We are now able to prove Lemma \ref{lembsp2} and Lemma \ref{lembsp}.

%%%%%%%%%%%%%%%%%%%%%%%%%%%%%%%%%%

  \begin{proof}[{\it Proof of Lemma~\ref{lembsp2}.}]
      A direct computation using the linearity of $\chi$ shows that,
    if $a_i,b_i$  are the vertices of the $i$-th subdomain edge $E^i$
    of $\omk$, we have
     $$
 | \chi^\ell |_{H^{1/2}(\partial\omk)}^2  \lesssim \sum_{i=1}^{\NE}
 (\eta^\ell(a_i) - \eta^\ell(b_i))^2$$
with $\NE$($=3$ or $4$) denoting the number of subdomain edges of $\omk$.
 Now, using~ \nref{eq:19} %  Lemma~\ref{lemmainf}-{\it (i)}
 and assembling all the contributions we easily conclude that the thesis
 holds.

\end{proof}

%%%%%%%%%%%%%%%%%%%%%%%%%%%
\begin{proof}[{\it Proof of Lemma~\ref{lembsp}}]
 Let $\zeta_0 \in \Tk$ be the unique element of $\Tk$
satisfying $\zeta_0(a) = 0$  for all vertices $a$  of $\omk$ and
$(\zeta_0,\tau)_{H^{1/2}(\partial\omk)} = (\zeta_L,\tau)_{H^{1/2}(\partial\omk)}$
for all $\tau \in \Tk$ with  $\tau(a) = 0$  for all vertices  $a$ of
$\omk$.
%, where $(\cdot,\cdot)_{H^{1/2}(\partial\omk)}$ is defined as in \nref{eq:scal12} with $s=1/2$.
It is not difficult to see that  $| \cdot
|_{H^{1/2}(\partial\omk)}$ is a norm on the subspace of functions in $\Tk$
vanishing at the vertices of $\omk$ and then, by standard arguments we get
that $\zeta_0$ is well defined and
$| \zeta_0 |_{H^{1/2}(\partial\omk)} \lesssim | \zeta_L |_{H^{1/2}(\partial\omk)} $.
  Now we can write:
  \begin{equation}
    \label{eq:6}
    \sum_{i=1}^{\NE} \| \xi \|^2_{H^{1/2}_{00}(\gki)} \lesssim
   \sum_{i=1}^{\NE}  \| \xi + \ziz\|^2_{H^{1/2}_{00}(\gki)} +
    \sum_{i=1}^{\NE} \| \ziz \|^2_{H^{1/2}_{00}(\gki)},
  \end{equation}
with $\NE$ number of subdomain edges of $\omk$.
The first sum on the right hand side of (\ref{eq:6}) can be bound by using the previous lemma as
\begin{align*}
  \sum_{i=1}^{\NE} \| \xi + \ziz\|^2_{H^{1/2}_{00}(\gki)} &\lesssim
  \faclogi^2 | \xi + \ziz |^2_{H^{1/2}(\partial\omk)} \\
&  \lesssim \faclogi^2 | \xi + \lin |^2_{H^{1/2}(\partial\omk)},
\end{align*}
where on one hand we used Poincar{\'e} inequality to bound the $H^{1/2}$
norm of $\xi + \ziz$ (which vanishes at the vertices of $\omk$) by the
corresponding seminorm, while the last inequality follows by observing that,
by the definition
of $\ziz$,
$\xi + \ziz \in \Tk$ vanishes at the vertices of
$\omk$ and satisfies $(\lin - \ziz,\xi+\ziz)_{H^{1/2}(\partial\omk)} = 0$.
Hence, we have
 $$| \xi + \lin |^2_{H^{1/2}(\partial\omk)} =
       | \xi + \ziz |^2_{H^{1/2}(\partial\omk)}+ |  \lin - \ziz |^2_{H^{1/2}(\partial\omk)} \geq   | \xi + \ziz |^2_{H^{1/2}(\partial\omk)}. $$
Let us now bound the second sum on the right hand side of (\ref{eq:6}):
we first observe that
\begin{equation*}
\| \ziz \|^2_{H^{1/2}_{00}(\gki)} = | \ziz |^2_{H^{1/2}(\gki)}+\Iuno(\ziz) + \Idue(\ziz),
\end{equation*}
having set
\begin{equation*}
\begin{aligned}
& \Iuno(\ziz) = \int_\aai^\bbi \frac {| \ziz(x) |^2} {|x - \aai|} \dd{x},
&&\Idue(\ziz) = \int_\aai^\bbi \frac {| \ziz(x) |^2} {|x - \bbi|} \dd{x},
\end{aligned}
\end{equation*}
with $\aai$ and $\bbi$ the two vertices of
the subdomain edge $\gki$. Now we can write
\begin{align*}
         \sum_{i=1}^{\NE}  | \ziz |^2_{H^{1/2}(\gki)} & \lesssim
         | \ziz |^2_{H^{1/2}(\partial\omk)} \lesssim
         | \lin |^2_{H^{1/2}(\partial\omk)} \lesssim
         \sum_{i=1}^{\NE} (\lin(\aai) - \lin(\bbi))^2 \\
         & \lesssim
\faclogi
% \left(1 + \log \frac H h\right)
%\left(1+\log{\left(\frac{H\po^{2}}{\ho}\right)}\right)
| \xi + \lin |^2_{H^{1/2}(\partial\omk)},
\end{align*}
where the inequality $ | \lin |^2_{H^{1/2}(\partial\omk)}\lesssim \sum_{i=1}^{\NE}
(\lin(\aai) - \lin(\bbi))^2$ is proven in \cite{BPS} by direct computation,
and the last inequality follows by applying the bound of Lemma~\ref{lemmainf}-{\it (ii)} to the function $(\xi + \lin)(x) - (\xi + \lin)(\bbi)$.

Let us now bound $\Iuno$. For notational simplicity let us identify $\aai = 0$ and $\bbi = H$. Adding and subtracting $\lin(x) + \lin(0)$ and using the Cauchy-Schwarz inequality, we have
\begin{equation}\label{eq:esti_I1}
\Iuno(\ziz) =
\int_0^H \frac {|\ziz(x)|^2} {|x|} \dd{x}
\lesssim
\int_0^H \frac {|\ziz(x) - \lin(x) + \lin(0)|^2} {|x|} \dd{x}
+\int_0^H \frac {|\lin(x) - \lin(0)|^2} {|x|} \dd{x}.
\end{equation}
\newcommand{\ziort}{\zeta_\perp}
Let us bound the first integral  on the right hand side of \nref{eq:esti_I1}. Setting $\ziort = \ziz - \lin$, we have
\[
\int_0^H \frac {| \ziort(x) - \ziort(0) |^2} {|x|} \dd{x} =
\int_0^h \frac {| \ziort(x) - \ziort(0) |^2} {|x|} \dd{x} +
\int_h^H \frac {| \ziort(x) - \ziort(0) |^2} {|x|} \dd{x}.\]
The first term can be bounded by
\begin{eqnarray*}
\int_0^h \frac {| \ziort(x) - \ziort(0) |^2} {|x|} \dd{x} =
\int_0^h \frac {| \int_0^x (\ziort)_x(\tau) \,d\tau |^2} {|x|} \dd{x}
\lesssim
h | \ziort |^2_{H^1(\gki)} \lesssim | \ziort |^2_{H^{1/2}(\gki)},
\end{eqnarray*}
while we bound the second term by
\begin{align*}
\int_h^H \frac {| \ziort(x) - \ziort(0) |^2} {|x|} \dd{x}
\lesssim &\| \ziort - \ziort(0) \|^2_{L^{\infty}(\gki)}
\log \left( \frac{\Ho\,\po^2}{\ho}\right) \\
\lesssim &
\left( \log  \left( \frac{\Ho\,\po^2}{\ho}\right) \right)^2 | \ziort |^2_{H^{1/2}(\gki)}.
\end{align*}
  Next, we estimate the second integral on the right hand side of \nref{eq:esti_I1}. By direct calculation and using the linearity of $\lin$, we have
\[
\int_0^H \frac {|\lin(x) - \lin(0)|^2} {|x|} \dd{x} \lesssim (\lin(\bi) - \lin(\ai))^2
\lesssim
\log  \left( \frac{\Ho\,\po}{\ho}  \right)\,|  \lin + \xi |^2_{H^{1/2}(\gki)}.
\]
Hence, we conclude that
 $$ \Iuno(\ziz) \lesssim \faclogi^2
%\left( 1 + \left( \log  \frac{H_\ell}{h_\ell} \right)^2 \right)
    |  \ziz - \lin  |^2_{H^{1/2}(\gki)} +  \log \left(
    \frac{\Ho\,\po^2}{\ho}  \right)\,|  \lin + \xi  |^2_{H^{1/2}(\gki)}.
$$
 The term $\Idue$ can be bounded by the same argument. Collecting all the previous estimates the thesis follows.
\end{proof}

% put your thanks here
\section*{Acknowledgments}
The work of P.F. Antonietti and B. Ayuso de Dios was partially
supported by Azioni Integrate Italia-Spagna through the projects
IT097ABB10 and HI2008-0173. 
B. Ayuso de Dios was also partially supported by grants  MINECO MTM2011-27739-C04-04
 and GENCAT 2009SGR-345.
Part of this work was done during several visits of B. Ayuso de Dios  to the Istituto {\it Enrico Magenes} IMATI-CNR at Pavia (Italy). She thanks the IMATI for the everlasting kind hospitality.  

%\bibliographystyle{plain}
%\bibliography{biblio}
%%%%%%%%%%%%%%%%%%%%%%%%%%%%%%%%%%%%%%%%%%%%%%%%%%%%%%%%%%%%%%%%%%%%%%

\end{document}